\documentclass[10pt,oneside,english]{amsart}
\usepackage[T1]{fontenc}
\usepackage[latin9]{inputenc}
\usepackage{geometry}
\usepackage{units}
\usepackage{enumitem}
\usepackage{amstext}
\usepackage{amsthm}
\usepackage{amssymb}
\usepackage{graphicx}

\makeatletter

\newcommand{\noun}[1]{\textsc{#1}}

\numberwithin{equation}{section}
\numberwithin{figure}{section}
\theoremstyle{plain}
\newtheorem{thm}{\protect\theoremname}[section]
  \theoremstyle{definition}
  \newtheorem{defn}[thm]{\protect\definitionname}
  \theoremstyle{plain}
  \newtheorem{fact}[thm]{\protect\factname}
  \theoremstyle{plain}
  \newtheorem{lem}[thm]{\protect\lemmaname}
  \theoremstyle{plain}
  \newtheorem{prop}[thm]{\protect\propositionname}
  \theoremstyle{remark}
  \newtheorem{claim}[thm]{\protect\claimname}

\makeatother

\usepackage{babel}
  \providecommand{\claimname}{Claim}
  \providecommand{\definitionname}{Definition}
  \providecommand{\factname}{Fact}
  \providecommand{\lemmaname}{Lemma}
  \providecommand{\propositionname}{Proposition}
\providecommand{\theoremname}{Theorem}

\begin{document}

\title{Komlós's tiling theorem via graphon covers}

\author{Jan Hladký}

\address{\emph{Current address: }Institute for geometry, TU Dresden, 01062
Dresden, Germany.\emph{ This work has been carried out while at: }Institute
of Mathematics, Czech Academy of Science. Žitná 25, 110~00, Praha,
Czech Republic. The Institute of Mathematics of the Czech Academy
of Sciences is supported by RVO:67985840.}

\thanks{\emph{Hladký:} The research leading to these results has received
funding from the People Programme (Marie Curie Actions) of the European
Union's Seventh Framework Programme (FP7/2007-2013) under REA grant
agreement number 628974.}

\email{honzahladky@gmail.com}

\author{Ping Hu}

\address{\emph{Current address:} School of Mathematics, Sun Yat-sen University,
Guangzhou, 510275, China. \emph{This work has been carried out while
at:} Department of Computer Science, University of Warwick, Coventry,
CV4 7AL, United Kingdom.}

\thanks{\emph{Hu:} This work has received funding from the Leverhulme Trust
2014 Philip Leverhulme Prize of Daniel Kral and from the European
Research Council (ERC) under the European Union\textquoteright s Horizon
2020 research and innovation programme (grant agreement No 648509).}

\email{huping9@mail.sysu.edu.cn}

\author{Diana Piguet}

\address{The Czech Academy of Sciences, Institute of Computer Science, Pod
Vodárenskou v\v{e}ží 2, 182~07 Prague, Czech Republic. With institutional
support RVO:67985807.}

\thanks{\emph{Piguet} was supported by the Czech Science Foundation, grant
number GJ16-07822Y}

\email{piguet@cs.cas.cz}
\begin{abstract}
Komlós {[}Komlós: Tiling Turán Theorems, Combinatorica, 2000{]} determined
the asymptotically optimal minimum-degree condition for covering a
given proportion of vertices of a host graph by vertex-disjoint copies
of a fixed graph $H$, thus essentially extending the Hajnal\textendash Szemerédi
theorem which deals with the case when $H$ is a clique. We give a
proof of a graphon version of Komlós's theorem. To prove this graphon
version, and also to deduce from it the original statement about finite
graphs, we use the machinery introduced in {[}Hladký, Hu, Piguet:
Tilings in graphons, arXiv:1606.03113{]}. We further prove a stability
version of Komlós's theorem.
\end{abstract}

\maketitle
\global\long\def\JUSTIFY#1{\mbox{\fbox{\tiny#1}}\quad}

\global\long\def\SUPPORT{\mathrm{supp\:}}

\global\long\def\SUPPORTPOSITIVE{\mathrm{supp}}

\global\long\def\ESSINF{\mathrm{essinf\:}}
\global\long\def\ESSSUP{\mathrm{esssup\:}}

\global\long\def\TIL{\mathsf{til}}

\global\long\def\FTIL{\mathsf{ftil}}

\global\long\def\FCOV{\mathsf{fcov}}

\global\long\def\DIST{\mathrm{dist}}

\global\long\def\EXP{\mathbf{E}}

\global\long\def\WEAKCONV{\overset{\mathrm{w}^{*}}{\;\longrightarrow\;}}

\global\long\def\CUTNORMCONV{\overset{\|\cdot\|_{\square}}{\;\longrightarrow\;}}

\section{Introduction\label{sec:Intro}}

Questions regarding the number of vertex-disjoint copies of a fixed
graph $H$ that can be found in a given graph $G$ are an important
part in extremal graph theory. The corresponding quantity, i.e., the
maximum number of vertex-disjoint copies of $H$ in $G$, is denoted
$\TIL(H,G)$, and called the \emph{tiling number of }$H$\emph{ in
$G$}. The by far most important case is when $H=K_{2}$ because then
$\TIL(H,G)$ is the matching number of $G$. For example, a classical
theorem of Erd\H{o}s\textendash Gallai~\cite{Erdos1959} gives an
optimal lower bound on the matching ratio of a graph in terms of its
edge density.

Recall that the theory of dense graph limits (initiated in~\cite{Lovasz2006,Borgs2008c})
and the related theory of flag algebras (introduced in~\cite{Razborov2007})
have led to breakthroughs on a number of long-standing problems that
concern relating subgraph densities. It is natural to attempt to broaden
the toolbox available in the graph limits world to be able to address
extremal problems that involve other parameters than subgraph densities.
In~\cite{HlHuPi:TilingsInGraphons} we worked out such a set of tools
for working with tiling numbers. In this paper we use this theory
to prove a strengthened version of a tiling theorem of Komlós,~\cite{Komlos2000}.

\subsection{Komlós's Theorem}

Suppose that $H$ is a fixed graph with chromatic number $r$. We
want to find a minimum degree threshold that guarantees a prescribed
lower bound on $\TIL(H,G)$ for a given (large) $n$-vertex graph
$G$. Consider first the special case $H=K_{r}$. Then one end of
the range for the problem is covered by Turán's Theorem: if $\delta(G)>\nicefrac{(r-2)n}{r-1}$
then $\TIL(H,G)\ge1$. The other end is covered by the Hajnal\textendash Szemerédi
Theorem, \cite{Hajnal1970}: if $\delta(G)\ge\left\lfloor \nicefrac{(r-1)n}{r}\right\rfloor $
then $\TIL(H,G)=\left\lfloor \nicefrac{n}{r}\right\rfloor $ (which
is the maximum possible value for $\TIL(H,G)$).\noun{ }If $\delta(G)=m<\left\lfloor \nicefrac{(r-1)n}{r}\right\rfloor $,
then we apply Hajnal-Szemerédi Theorem to the complement of $G$ to
get an equitable coloring with $n-m+1$ colors, such that the size
of color classes are $r$or $r-1$. And therefore we get $\TIL(H,G)=n-(r-1)(n-m+1)$.

When $H$ is a general $r$-chromatic graph, the asymptotically optimal
minimum degree condition $\delta(G)\ge(1+o_{n}(1))\nicefrac{(r-1)n}{r}$
for the property $\TIL(H,G)\ge1$ is given by the Erd\H{o}s\textendash Stone
Theorem (see Section~\ref{subsec:ErdStoSim}). Komlós's Theorem then
determines the optimal threshold for greater values of $\TIL(H,G)$.
To this end we need to introduce the critical chromatic number\emph{.}
\begin{defn}
Suppose that $H$ is a graph of order $h$ whose chromatic number
is $r$. We write $\ell$ for the order of the smallest possible color
class in any $r$-coloring of $H$. The \emph{critical chromatic number
}of $H$ is then defined as
\begin{equation}
\chi_{\mathrm{cr}}(H)=\frac{(r-1)h}{h-\ell}\;.\label{def:chiCR}
\end{equation}
\end{defn}

Observe that 
\begin{equation}
\chi_{\mathrm{cr}}(H)\in\left(\chi(H)-1,\chi(H)\right]\;.\label{eq:chicrVSchi}
\end{equation}
We can now state Komlós's Theorem.
\begin{thm}[\cite{Komlos2000}]
\label{thm:KomlosOriginal}Let $H$ be an arbitrary graph, and $x\in[0,1]$.
Then for every $\epsilon>0$ there exists a number $n_{0}$ such that
the following holds. Suppose that $G$ is a graph of order $n>n_{0}$
with minimum degree at least 
\begin{equation}
\left(x\left(1-\frac{1}{\chi_{\mathrm{cr}}(H)}\right)+\left(1-x\right)\left(1-\frac{1}{\chi(H)-1}\right)\right)n\;.\label{eq:KomlosOrigMinDeg}
\end{equation}
Then $\TIL(H,G)\ge\frac{\left(x-\epsilon\right)n}{v(H)}$.
\end{thm}

This result is tight (up to the error term $\frac{\epsilon n}{v(H)}$)
as shown by an $\chi(H)$-partite $n$-vertex graph whose $\chi(H)-1$
colour classes are of size $\nicefrac{n\cdot\left(\chi(H)-x\left(\chi_{\mathrm{cr}}(H)+1-\chi(H)\right)\right)}{\chi(H)\left(\chi(H)-1\right)}$
each, and the $\chi(H)$-th colour class is of size $n\cdot\nicefrac{x\left(\chi_{\mathrm{cr}}(H)+1-\chi(H)\right)}{\chi(H)}$.\footnote{Again, we neglect rounding issues.}
Additional edges can be inserted into the last colour class arbitrarily.
Komlós calls these graphs \emph{bottleneck graphs with parameters
$x$ and $\chi_{\mathrm{cr}}(H)$}.\footnote{Note that the parameter $\chi(H)$ need not be an input as it can
be reconstructed from $\chi_{\mathrm{cr}}(H)$ using~(\ref{eq:chicrVSchi}).}

Note also that Theorem~\ref{thm:KomlosOriginal} does not cover the
case of perfect tilings, i.e., when $\TIL(H,G)=\left\lfloor \frac{n}{v(H)}\right\rfloor $.
Indeed, the answer to this ``exact problem'' (as opposed to approximate)
is more complicated as was shown by Kühn and Osthus~\cite{Kuehn2009}. 

Here, we reprove Komlós's Theorem. Actually, our proof also gives
a stability version of Theorem~\ref{thm:KomlosOriginal}. This stability
version seems to be new.
\begin{thm}
\label{thm:KomlosStability}Let $H$ be an arbitrary graph, and $x\in[0,1]$.
Then for every $\epsilon>0$ there exists a number $n_{0}\in\mathbb{N}$
such that the following holds. Suppose that $G$ is a graph of order
$n>n_{0}$ with minimum degree at least as in~(\ref{eq:KomlosOrigMinDeg}).
Then 
\[
\TIL(H,G)\ge\frac{\left(x-\epsilon\right)n}{v(H)}\;.
\]

Furthermore, if $x\in[0,1)$ then for every $\epsilon>0$ there exist
numbers $n_{0}\in\mathbb{N}$ and $\delta>0$ such that the following
holds. Suppose that $G$ is a graph of order $n>n_{0}$ with minimum
degree at least as in~(\ref{eq:KomlosOrigMinDeg}). Then we have
\[
\TIL(H,G)\ge\frac{\left(x+\delta\right)n}{v(H)}\;,
\]
unless $G$ is $\epsilon$-close in the edit distance\footnote{see Section~\ref{subsec:EditDistance} for a definition}
to a bottleneck graph with parameters $x$ and $\chi_{\mathrm{cr}}(H)$.
\end{thm}

The original proof of Theorem~\ref{thm:KomlosOriginal} is not lengthy
but uses an ingenious recursive regularization of the graph $G$.\footnote{See Section~\ref{sec:ComparingProofs}.}
Our proof offers an alternative point of view on the problem. In fact
we believe it follows the most natural strategy: \emph{If $G$ had
only a small tiling number then, by the LP duality,}\footnote{Normally, the LP duality would require the \emph{fractional} version
of the tiling number to be considered. However, we are able to overcome
this matter.}\emph{ it would have a small fractional $F$-cover. This would lead
to a contradiction to the minimum degree assumption.} The actual execution
of this proof strategy, using the graphon formalism, is quite technical,
in particular in the stability part. Tools that we need to use to
this end involve the Banach\textendash Alaoglu Theorem, and arguments
about separability of function spaces. While the amount of analytic
tools needed may be viewed as a disincentive we actually believe that
working out these techniques will be useful in bringing more tools
from graph limit theories to extremal combinatorics.

\subsection{Organization of the paper}

In Section~\ref{sec:Preliminaries} we introduce the notation and
recall background regarding measure theory, graphons and extremal
graph theory. In Section~\ref{sec:TilingsGraphons} we give a digest
of those parts of the theory of tilings in graphons developed in~\cite{HlHuPi:TilingsInGraphons}
that are needed in the present paper. Thus, any reader familiar with
the general theory of graphons should be able to read this paper without
having to study~\cite{HlHuPi:TilingsInGraphons}. In Section~\ref{sec:KomlosThm}
we state the graphon version of Komlós's Theorem, and use it to deduce
Theorem~\ref{thm:KomlosStability}. This graphon version of Komlós's
Theorem is then proved in Section~\ref{sec:ProofKomlos}. Sections~\ref{sec:ComparingProofs}
and~\ref{sec:FurtherApplications} contain some concluding comments.

\section{Preliminaries\label{sec:Preliminaries}}

\subsection{Basic measure theory and weak{*} convergence}

Throughout, we shall work with an atomless Borel probability space
$\Omega$ equipped with a measure $\nu$ (defined on an implicit $\sigma$-algebra). 

Given a function $f$ and a number $a$ we define its support $\SUPPORT f=\{x:f(x)\neq0\}$
and its variant $\SUPPORTPOSITIVE_{a}\:f=\{x:f(x)\ge a\}$. Recall
that a set is \emph{null} if it has zero measure. \textquotedblleft \emph{Almost
everywhere}\textquotedblright{} is a synonym to \textquotedblleft up
to a null-set\textquotedblright . If $f$ is a measurable function,
we write $\ESSINF f:=\textrm{sup}\{a:\{f(x)\le a\}\textrm{ is null}\}$
for the \emph{essential infimum} of $f$ and $\ESSSUP f:=\textrm{inf}\{a:\{f(x)\ge a\}\textrm{ is null}\}$
for the \emph{essential supremum} of $f$.

The product measure on $\Omega^{k}$ is denoted by $\nu^{k}$. Recall,
that this measure can be constructed by Caratheodory's construction
from the $k$-th power of the $\sigma$-algebra underlying $\Omega$.
In particular, we have the following basic fact (which we state only
for the case $k=2$, which will be needed later).
\begin{fact}
\label{fact:approximate}Suppose that $P\subset\Omega^{2}$ is a set
of positive measure. Then for every $\epsilon>0$ there exist sets
$X,Y\subset\Omega$ of positive measure so that
\[
\nu^{2}\left(X\times Y\cap P\right)\ge(1-\epsilon)\nu(X)\nu(Y)\;.
\]
\end{fact}

If $\Omega$ is a Borel probability space, then it is a separable
measure space. The Banach space $\mathcal{L}^{1}(\Omega)$ is separable
(see e.g.~\cite[Theorem 13.8]{BrBrTh:RealAnalysis}). The dual of
$\mathcal{L}^{1}(\Omega)$ is $\mathcal{L}^{\infty}(\Omega)$. Recall
that a sequence $f_{1},f_{2},\ldots\in\mathcal{L}^{\infty}(\Omega)$
\emph{converges weak{*}} to a function $f\in\mathcal{L}^{\infty}(\Omega)$
if for each $g\in\mathcal{L}^{1}(\Omega)$ we have that $\int_{\Omega}f_{n}g\rightarrow\int_{\Omega}fg$.
This convergence notion defines the so-called \emph{weak{*} topology}
on $\mathcal{L}^{\infty}(\Omega)$. Let us remark that this topology
is not metrizable in general. The sequential Banach\textendash Alaoglu
Theorem (as stated for example in~\cite[Theorem 1.9.14]{Tao:EpsilonRoomI})
in this setting reads as follows.
\begin{thm}
\label{thm:BanachAlauglou}If $\Omega$ is a Borel probability space
then each sequence of functions of $\mathcal{L}^{\infty}(\Omega)$-norm
at most~1 contains a weak{*} convergent subsequence.
\end{thm}

\subsection{Graphons}

Our notation follows mostly~\cite{Lovasz2012}. Our graphons will
be defined on $\Omega^{2}$. Recall that $\Omega$ is an atomless
Borel probability space with probability measure $\nu$. 

We refer the reader to~\cite{Lovasz2012} to the key notions of \emph{cut-norm
}$\|\cdot\|_{\square}$ and \emph{cut-distance }$\DIST_{\square}(\cdot,\cdot)$.
We just emphasize that to derive the latter from the former, one has
to involve certain measure-preserving bijections. This step causes
that the cut-distance is coarser (in the sense of topologies) than
then cut-norm. When we say that a sequence of graphs \emph{converges}
to a graphon we refer to the cut-distance.

Suppose that we are given an arbitrary graphon $W:\Omega^{2}\rightarrow[0,1]$
and a graph $F$ whose vertex set is $[k]$. We write $W^{\otimes F}:\Omega^{k}\rightarrow[0,1]$
for a function defined by 
\[
W^{\otimes F}(x_{1},\ldots,x_{k})=\prod_{\substack{1\le i<j\le k\\
ij\in E(F)
}
}W(x_{i},x_{j})\;.
\]

Last, let us recall the notion of neighborhood and degree in a graphon
$W:\Omega^{2}\rightarrow[0,1]$. If $x_{1},\ldots,x_{\ell}\in\Omega$,
then the \emph{common neighborhood} $N(x_{1},\ldots,x_{\ell})$ is
the set $\bigcap_{i=1}^{\ell}\left(\SUPPORT W(x_{i},\cdot)\right)$.
The \emph{degree} of a vertex $x\in\Omega$ is $\deg_{W}(x)=\int_{y\in\Omega}W(x,y)$.
The \emph{minimum degree of $W$} is $\delta(W)=\ESSINF\deg_{W}(x)$.
It is well-known (see for example~\cite[Theorem 3.15]{Razborov2007})
that any limit graphon of sequence of graphs with large minimum degrees
has a large minimum degree.
\begin{lem}
\label{lem:mindegsemicontinuous}Suppose $\alpha>0$ and that $G_{1},G_{2},\ldots$
are finite graphs converging to a graphon $W$, and that their minimum
degrees satisfy $\delta(G_{i})\ge\alpha v(G_{i})$. Then $\delta(W)\ge\alpha$.\qed
\end{lem}

\subsection{Independent sets in graphons}

If $W:\Omega^{2}\rightarrow[0,1]$ is a graphon then we say that a
measurable set $A\subset\Omega$ is an \emph{independent set} in $W$
if $W$ is~0 almost everywhere on $A\times A$. The next (standard)
lemma asserts that a weak{*} limit of independent sets is again an
independent set.
\begin{lem}
\label{lem:limitofindependent}Let $W:\Omega^{2}\rightarrow[0,1]$
be a graphon. Suppose that $\left(A_{n}\right)_{n=1}^{\infty}$ is
a sequence of independent sets in $W$. Suppose that the indicator
functions of the sets $A_{n}$ converge weak{*} to a function $f:\Omega\rightarrow[0,1]$.
Then $\SUPPORT f$ is an independent set in $W$.
\end{lem}

\begin{proof}
It is enough to prove that for each $\epsilon>0$, the set $P=\SUPPORTPOSITIVE_{\epsilon}f$
is independent. There is nothing to prove if $P$ is null, so assume
that $P$ has positive measure. Suppose that the statement is false.
Then by by Fact~\ref{fact:approximate} there exist sets $X,Y\subset P$
of positive measure such that
\begin{equation}
\nu^{2}\left(X\times Y\cap\left\{ (x,y)\in\Omega^{2}:W(x,y)=0\right\} \right)<\frac{\epsilon^{2}}{5}\nu(X)\nu(Y)\;.\label{eq:epsXY}
\end{equation}
Recall that $\int_{X}f\ge\epsilon\nu(X)$ and $\int_{Y}f\ge\epsilon\nu(Y)$.
By weak{*} convergence, for $n$ sufficiently large, $\nu(X\cap A_{n})\ge\frac{\epsilon}{2}\nu(X)$
and $\nu(Y\cap A_{n})\ge\frac{\epsilon}{2}\nu(Y)$. Since $A_{n}$
is an independent set, we have that $W$ is~0 almost everywhere on
$\left(X\cap A_{n}\right)\times\left(Y\cap A_{n}\right)$. This contradicts~(\ref{eq:epsXY}).
\end{proof}

\subsection{Edit distance\label{subsec:EditDistance}}

Given two $n$-vertex graphs $G$ and $H$, the \emph{edit distance
from $G$ to $H$} is the number of edges of $G$ that need to be
edited (i.e., added or deleted) to get $H$ from $G$. Here, we minimize
over all possible identifications of $V(G)$ and $V(H)$. So, for
example if $G$ and $H$ are isomorphic then their edit distance is~0.
We say that $H$ is \emph{$\epsilon$-close} to $G$ in the edit distance
if its distance from $H$ is at most $\epsilon{n \choose 2}$.

\subsection{Erd\H{o}s\textendash Stone\textendash Simonovits Stability Theorem\label{subsec:ErdStoSim}}

Suppose that $H$ is a graph of chromatic number $r$. The Erd\H{o}s\textendash Stone\textendash Simonovits
Stability Theorem~\cite{Erdos1946,Simonovits1968} asserts that if
$G$ is an $H$-free graph on $n$ vertices then $e(G)\le\left(1-\frac{1}{r-1}+o_{n}(1)\right){n \choose 2}$.
This is accompanied by a stability statement: for each $\epsilon>0$
there exists numbers $\delta>0$ and $n_{0}$ such that if $G$ is
an $H$-free graph on $n$ vertices, $n>n_{0}$ and $e(G)>\left(1-\frac{1}{r-1}-\delta\right){n \choose 2}$,
then $G$ must be $\epsilon$-close to the $(r-1)$-partite Turán
graph in the edit distance. We shall need the min-degree version of
this (which is actually weaker and easier to prove): if the minimum
degree of $G$ is at least $\left(1-\frac{1}{r-1}-\delta\right)n$
and $G$ is $H$-free, then $G$ must be $\epsilon$-close to the
$(r-1)$-partite Turán graph in the edit distance.

We say that $W:\Omega^{2}\rightarrow[0,1]$ is a \emph{$(r-1)$-partite
Turán graphon} if there exists a partition $\Omega=\Omega_{1}\sqcup\ldots\sqcup\Omega_{r-1}$
into sets of measure $\nicefrac{1}{r-1}$ each, such that $W_{\restriction\Omega_{i}\times\Omega_{j}}$
equals~1 almost everywhere for $i\neq j$ and equals~0 almost everywhere
for $i=j$. The stability part of the min-degree version of the Erd\H{o}s\textendash Stone\textendash Simonovits
Theorem yields the following:
\begin{thm}
\label{thm:graphonErdosStoneSimonovits}Suppose that $H$ is a graph
of chromatic number $r$. If $W$ is a graphon with $\int_{\Omega^{V(H)}}W^{\otimes H}=0$
and minimum degree at least $1-\frac{1}{r-1}$, then $W$ is a $(r-1)$-partite
Turán graphon.
\end{thm}

\section{Tilings in graphons\label{sec:TilingsGraphons}}

In this section, we recall the main concepts and results from~\cite{HlHuPi:TilingsInGraphons}.
Let us first recall the most important definitions of an \emph{$F$-tiling}
and a \emph{fractional $F$-cover} in a graphon. The definition of
$F$-tilings in graphons is inspired by the definition of \emph{fractional}
$F$-tilings in finite graphs (we explained in~\cite[Section 3.2]{HlHuPi:TilingsInGraphons}
that there should be no difference between integral and fractional
$F$-tilings in graphons). 
\begin{defn}
\label{def:graphontiling}Suppose that $W:\Omega^{2}\rightarrow[0,1]$
is a graphon, and that $F$ is a graph on the vertex set~$[k]$.
A function $\mathfrak{t}:\Omega^{k}\rightarrow[0,+\infty)$ is called
an \emph{$F$-tiling} in $W$ if 
\[
\text{\ensuremath{\SUPPORT\mathfrak{t}\subset}}\SUPPORT W^{\otimes F}\;,
\]
 and we have for each $x\in\Omega$ that 
\[
\sum_{\ell=1}^{k}\int_{(x_{1},\ldots,x_{\ell-1},x_{\ell+1},\ldots,x_{k})\in\Omega^{k-1}}\mathfrak{t}(x_{1},\ldots,x_{\ell-1},x,x_{\ell+1},\ldots,x_{k})\le1\;.
\]
The \emph{size} of an $F$-tiling $\mathfrak{t}$ is $\|\mathfrak{t}\|=\int_{\Omega^{k}}\mathfrak{t}$.
The \emph{$F$-tiling number} of $W$, denoted by $\TIL(F,W)$, is
the supremum of sizes over all $F$-tilings in $W$.
\end{defn}

For the definition of fractional $F$-covers in graphons one just
rewrites \emph{mutatis mutandis} the usual axioms of fractional $F$-covers
in finite graphs.
\begin{defn}
\label{def:covergraphon}Suppose that $W:\Omega^{2}\rightarrow[0,1]$
is a graphon, and~$F$ is a graph on the vertex set~$[k]$. A measurable
function $\mathfrak{c}:\Omega\rightarrow[0,1]$ is called a \emph{fractional
$F$-cover} in $W$ if
\[
\nu^{k}\left(\left(\SUPPORT W^{\otimes F}\right)\cap\left\{ (x_{1},x_{2},\ldots,x_{k})\in\Omega^{k}:\sum_{i=1}^{k}\mathfrak{c}(x_{i})<1\right\} \right)=0\;.
\]
The \emph{size} of $\mathfrak{c}$, denoted by $\|\mathfrak{c}\|$,
is defined by $\|\mathfrak{c}\|=\int_{\Omega}\mathfrak{c}$. The \emph{fractional
$F$-cover number }$\FCOV(F,W)$ of $W$ is the infimum of the sizes
of fractional $F$-covers in $W$\emph{.}
\end{defn}

Let us note that in~\cite[(3.7)]{HlHuPi:TilingsInGraphons}, we established
that
\begin{equation}
\mbox{the value of \ensuremath{\FCOV(F,W)} is attained by some fractional \ensuremath{F}-cover.}\label{eq:fcovattained}
\end{equation}

With these notions at hand, we can state two key results from~\cite{HlHuPi:TilingsInGraphons}:
the lower-semicontinuity of the $F$-tiling number, and the graphon
LP-duality.
\begin{thm}[{\cite[Theorem 3.4]{HlHuPi:TilingsInGraphons}}]
\label{thm:lowersemicontgraphs}Suppose that $F$ is a finite graph
and suppose that $(G_{n})$ is a sequence of graphs of growing orders
converging to a graphon $W:\Omega^{2}\rightarrow[0,1]$ in the cut-distance.
Then we have that $\liminf_{n}\frac{\TIL(F,G_{n})}{v(G_{n})}\ge\TIL(F,W)$.
\end{thm}

\begin{thm}[{\cite[Theorem 3.16]{HlHuPi:TilingsInGraphons}}]
\label{thm:LPdualityGraphons}Suppose that $W:\Omega^{2}\rightarrow[0,1]$
is a graphon and $F$ is an arbitrary finite graph. Then we have $\TIL(F,W)=\FCOV(F,W)$.
\end{thm}

\medskip{}

The following useful proposition relates qualitatively the $F$-tiling
number and the $F$-homomorphism density.
\begin{prop}
\label{prop:tilingVSdensity}Suppose that $F$ is a finite graph on
a vertex set $[k]$ . Then for an arbitrary graphon W we have that
$\TIL(F,W)=0$ if and only if 
\begin{equation}
\int_{\Omega^{k}}W^{\otimes F}=0\;.\label{eq:densityzero}
\end{equation}
\end{prop}

\begin{proof}
By Theorem~\ref{thm:LPdualityGraphons} and~(\ref{eq:fcovattained})
we know, that $\TIL(F,W)=0$ if and only if the constant zero function
(up to a null set) is a fractional $F$ -cover of $W$ . The latter
property is equivalent to~(\ref{eq:densityzero}).
\end{proof}

\section{Komlós's Theorem\label{sec:KomlosThm}}

We state our result as a graphon counterpart of Theorem~\ref{thm:KomlosOriginal}.
First, in analogy to bottleneck graphs we define the class of bottleneck
graphons.
\begin{defn}
\label{def:bottleneckgraphon}Suppose that numbers\emph{ $x\in[0,1)$
and $\chi_{\mathrm{cr}}\in(1,+\infty)$} are given. Let us write $r=\left\lceil \chi_{\mathrm{cr}}\right\rceil $.
We say that a graphon $W:\Omega^{2}\rightarrow[0,1]$ is a \emph{bottleneck
graphon with parameters~$x$ and~$\chi_{\mathrm{cr}}$ }if there
exists a partition $\Omega=\Omega_{1}\sqcup\Omega_{2}\sqcup\ldots\sqcup\Omega_{r}$
such that $\nu(\Omega_{r})=\nicefrac{x\left(\chi_{\mathrm{cr}}+1-r\right)}{r}$,
$\nu(\Omega_{1})=\nu(\Omega_{2})=\ldots=\nu(\Omega_{r-1})=\nicefrac{\left(r-x\left(\chi_{\mathrm{cr}}+1-r\right)\right)}{r\left(r-1\right)}$,
and such that

\begin{itemize}
\item for each $1\le i<j\le r$, $W$ is 1 almost everywhere on $\Omega_{i}\times\Omega_{j}$,
\item for each $1\le i\le r-1$, $W$ is 0 almost everywhere on $\Omega_{i}\times\Omega_{i}$.
\end{itemize}
\end{defn}

A set of graphons on a given probability space $\Omega$ is called
\emph{a graphon class} if with each graphon it contains all graphons
isomorphic to it. Given a graphon $W$ and a graphon class $\mathcal{C}$,
we define $\DIST_{\square}(W,\mathcal{C})=\inf_{U\in\mathcal{C}}\|W-U\|_{\square}$.
We also define $\DIST_{1}(W,\mathcal{C})=\inf_{U\in\mathcal{C}}\|W-U\|_{1}$.

For a given $x\in[0,1]$ and $\chi_{\mathrm{cr}}\in(1,\infty)$, we
write $\mathcal{C}_{x,\chi_{\mathrm{cr}}}$ for the set of all bottleneck
graphons with parameters $x$ and $\chi_{\mathrm{cr}}$. This is obviously
a graphon class. The next standard lemma asserts that convergence
to $\mathcal{C}_{x,\chi_{\mathrm{cr}}}$in the cut-norm implies convergence
in the $\mathcal{L}^{1}$-norm.
\begin{lem}
\label{lem:bottleEditVsCut}Suppose that $x\in[0,1]$ and $\chi_{\mathrm{cr}}\in(1,\infty)$.
If $\left(W_{n}\right)$ is a sequence of graphons with $\DIST_{\square}(W_{n},\mathcal{C}_{x,\chi_{\mathrm{cr}}})\rightarrow0$
then $\DIST_{1}(W_{n},\mathcal{C}_{x,\chi_{\mathrm{cr}}})\rightarrow0$.
\end{lem}

\begin{proof}
Let $B_{x,\chi_{\mathrm{cr}}}$ be (any representative of the isomorphism
class of) the bottleneck graphons with parameters $x$ and $\chi_{\mathrm{cr}}$
in which $B_{x,\chi_{\mathrm{cr}}}$ restricted to $\Omega_{r}\times\Omega_{r}$
is zero. The fact that $\DIST_{\square}(W_{n},\mathcal{C}_{x,\chi_{\mathrm{cr}}})\rightarrow0$
allows us to find partitions $\Omega^{(n)}=\Omega_{1}^{(n)}\sqcup\ldots\sqcup\Omega_{r}^{(n)}$
where the sets $\Omega_{i}^{(n)}$ have measures as in Definition~\ref{def:bottleneckgraphon}
and approximately satisfy the other properties. Let us modify each
graphon $W_{n}$ by making it zero on $\Omega_{r}^{(n)}\times\Omega_{r}^{(n)}$.
For the modified graphons $W'_{n}$, we have $\DIST_{\square}(W'_{n},B_{x,\chi_{\mathrm{cr}}})\rightarrow0$.
The graphon~$B_{x,\chi_{\mathrm{cr}}}$ is 0-1-valued. Thus,~\cite[Proposition 8.24]{Lovasz2012}
tells us that $\DIST_{1}(W'_{n},B_{x,\chi_{\mathrm{cr}}})\rightarrow0.$
Consequently, $\DIST_{1}(W{}_{n},\mathcal{C}_{x,\chi_{\mathrm{cr}}})\rightarrow0$.
\end{proof}
\begin{thm}
\label{thm:KomlosGraphon}Let $H$ be an arbitrary graph with chromatic
number at least two, and $x\in[0,1]$. Suppose that $W$ is a graphon
with minimum degree at least 
\begin{equation}
x\left(1-\frac{1}{\chi_{\mathrm{cr}}(H)}\right)+\left(1-x\right)\left(1-\frac{1}{\chi(H)-1}\right)\;.\label{eq:defdelta}
\end{equation}
Then $\FCOV(H,W)\ge\frac{x}{v(H)}$. Furthermore, if $x<1$ and $\FCOV(H,W)=\frac{x}{v(H)}$
then $W$ is a bottleneck graphon with parameters $x$ and $\chi_{\mathrm{cr}}:=\chi_{\mathrm{cr}}(H)$.\footnote{Clearly, there is no uniqueness for $x=1$.}
\end{thm}

The proof of Theorem~\ref{thm:KomlosGraphon} occupies Section~\ref{sec:ProofKomlos}.
Let us now employ the transference results from Section~\ref{sec:TilingsGraphons}
to see that Theorem~\ref{thm:KomlosGraphon} indeed implies Theorem~\ref{thm:KomlosStability}.
\begin{proof}[Proof of Theorem~\ref{thm:KomlosStability}]
We first prove the main assertion, and leave the ``furthermore''
part for later. Suppose that $(G_{n})_{n}$ is a sequence of graphs
with 
\begin{equation}
\delta(G_{n})\ge\left(x\left(1-\frac{1}{\chi_{\mathrm{cr}}(H)}\right)+\left(1-x\right)\left(1-\frac{1}{\chi(H)-1}\right)\right)v(G_{n})\label{eq:Gnmindeg}
\end{equation}
whose orders tend to infinity for some fixed $x>0$ and a finite graph
$H$. Let $W$ be a graphon that is an accumulation point of this
sequence with respect to the cut-distance. Then the minimum degree
of $W$ is at least $x\left(1-\frac{1}{\chi_{\mathrm{cr}}(H)}\right)+\left(1-x\right)\left(1-\frac{1}{\chi(H)-1}\right)$
by Lemma~\ref{lem:mindegsemicontinuous}. Thus Theorem~\ref{thm:KomlosGraphon}
tells us that $\FCOV(H,W)\ge\frac{x}{v(H)}$. Then Theorems~\ref{thm:lowersemicontgraphs}
and~\ref{thm:LPdualityGraphons} imply that $\liminf_{n}\frac{\TIL(H,G_{n})}{v(G_{n})}\ge\TIL(H,W)=\FCOV(H,W)$,
as needed. 

Let us now move to the ``furthermore'' part of the statement. Suppose
that $(G_{n})_{n}$ is a sequence of graphs whose orders tend to infinity
which satisfies~(\ref{eq:Gnmindeg}) for some fixed $x>0$ and a
finite graph $H$. Suppose that for each $\delta>0$, when $n$ is
sufficiently large, we have that $\TIL(H,G_{n})\le\frac{x+\delta}{v(H)}\cdot n$.
Let us now pass to any limit graphon $W$. We have $\delta(W)\ge x\left(1-\frac{1}{\chi_{\mathrm{cr}}(H)}\right)+\left(1-x\right)\left(1-\frac{1}{\chi(H)-1}\right)$
and, by Theorems~\ref{thm:lowersemicontgraphs} and~\ref{thm:LPdualityGraphons},
we have that $\TIL(H,W)\le\frac{x}{v(H)}$. Theorem~\ref{thm:KomlosGraphon}
tells us that $W$ must be a bottleneck graphon with parameters $x$
and $\chi_{\mathrm{cr}}(H)$. We conclude, that for large enough $n$,
the graph $G_{n}$ is $\epsilon$-close in the cut-distance to a bottleneck
graph with parameters $x$ and $\chi_{\mathrm{cr}}(H)$. Furthermore,
by Lemma~\ref{lem:bottleEditVsCut}, we can actually infer $\epsilon$-closeness
in the edit distance, as was needed.
\end{proof}

\section{Proof of Theorem~\ref{thm:KomlosGraphon}\label{sec:ProofKomlos}}

In Section~\ref{subsec:nonstabilityKomlos} we prove the main part
of the statement, and in Section~\ref{subsec:stabilityKomlos} we
refine our arguments to get the stability asserted in the ``furthermore''
part. Prior to each of these two section, an overview of the proof
is given.

Throughout the section, we shall work with ``slices of $W$'', i.e.,
one-variable functions $W(x,\cdot)$ for some fixed $x\in\Omega$.
Recall that measurability of $W(\cdot,\cdot)$ gives that $W(x,\cdot)$
is measurable for almost every $x\in\Omega$. We shall assume that
$W(x,\cdot)$ is measurable for every $x\in\Omega$. This is only
for the sake of notational simplicity; in the formal proofs we would
first take away the exceptional set of $x$'s.

Let us write $\delta=\delta(W)$.

Let us first deal with the case $x=0$. Then the only non-trivial
assertion in Theorem~\ref{thm:KomlosGraphon} is the stability. So,
suppose that the conditions of the theorem are fulfilled with $x=0$,
and we have $\FCOV(H,W)=0$. Then Theorem~\ref{thm:LPdualityGraphons}
and Proposition~\ref{prop:tilingVSdensity} tell us that $\int_{\Omega^{V(H)}}W^{\otimes H}=0$.
Recall that $\delta\ge1-\frac{1}{\chi(H)-1}$ by~(\ref{eq:defdelta}).
The Erd\H{o}s\textendash Stone\textendash Simonovits Stability Theorem~\ref{thm:graphonErdosStoneSimonovits}
tells us that $W$ must be a $\chi(H)$-partite Turán graphon. By
Definition~\ref{def:bottleneckgraphon}, this is equivalent to being
a bottleneck graphon with parameters $0$ and $\chi_{\mathrm{cr}}(H)$,
which was to be proven.

Thus, throughout the remainder of the proof, we shall assume that
$x$ is positive.

\subsection{Overview of the proof of the main part of the statement\label{subsec:overviewMain}}

Here, we provide an overview of the proof of the main part of Theorem~\ref{thm:KomlosGraphon}.
The proof itself, as written in Section~\ref{subsec:nonstabilityKomlos}
requires to deal with several technicalities stemming from our infinitesimal
approach to the problem (e.g., infima need not be attained). To separate
these technicalities from the key ideas, in this overview we shall
assume that $\Omega$ is a \emph{finite }probability space, $\Omega=\{\omega_{1},\ldots,\omega_{z}\}$.
(We shall assume that each $\omega_{j}$ has positive measure.) The
reader can then view $W$ as a finite cluster graph with ``clusters''
$\omega_{1},\ldots,\omega_{z}$. (The clusters are not required to
have the same size.) In this overview, we try to make use of this
analogy and explain the ideas behind our proof from the Regularity
lemma perspective. We essentially use the same notation as in Section~\ref{subsec:nonstabilityKomlos};
the only difference is that our objects are simpler due to the discrete
setting. That is, in the actual execution of the proof in Section~\ref{subsec:nonstabilityKomlos},
we will have to incorporate small additional error parameters to the
setting. We comment on the differences at the end of this overview.

Among all proper colourings of $H$ with $r=\chi(H)$ colours consider
one that minimizes the size of the smallest colour class and let $V(H)=V_{1}\sqcup V_{2}\sqcup\ldots\sqcup V_{r}$
be the partition of the vertex set into the colour classes of this
colouring such that $\ell_{1}\ge\ell_{2}\ge\ldots\ge\ell_{r}>0,$
for $\ell_{i}=|V_{i}|$. Let $h=\sum_{i}\ell_{i}$ be the order of~$H$.
Let $\mathfrak{c}:\Omega\rightarrow[0,1]$ be an arbitrary fractional
$H$-cover of $W$. Notice that Definition~\ref{def:covergraphon}
is consistent with the usual graph-theoretic definition of a fractional
cover when the target $W$ is viewed as a finite graph (``cluster
graph''). However, we emphasize that this corresponding graph-theoretic
definition of a fractional cover is about homomorphisms rather than
copies. That is, the requirement is that
\begin{equation}
\sum_{k=1}^{h}\mathfrak{c}(x_{k})\ge1\;,\label{eq:Piaf}
\end{equation}
whenever $\left\{ x_{v}\in\{\omega_{1},\ldots,\omega_{z}\}\right\} _{v\in V(H)}$
is an $h$-tuple of not necessarily different clusters with the property
that $W(x_{u},x_{v})>0$ for each $uv\in E(H)$. This definition makes
sense even if not all the clusters $x_{v}$ are distinct as regularity
embedding techniques allow us to embed $H$ into the corresponding
collection of clusters even in this setting.

We need to show that $\int_{\Omega}\mathfrak{c}\ge\frac{x}{v(H)}$.
To get such a lower-bound, we start focusing on those parts of $\Omega$
where the value of $\mathfrak{c}$ is small. More precisely, our idea
is to take a cluster $B_{1}\in\{\omega_{1},\ldots,\omega_{z}\}$ with
the smallest value of $\mathfrak{c}$. Then, having defined the clusters
$B_{1},\ldots,B_{i}$ (for some $i<r$), we take $B_{i+1}\in\{\omega_{1},\ldots,\omega_{z}\}$
to be the cluster that has the smallest value of $\mathfrak{c}$ in
the common neighborhood of $B_{1},\ldots,B_{i}$. Notice that since
our minimum-degree is bigger than $1-\frac{1}{r-1}$, these common
neighborhood are indeed nonempty. In particular, the clusters $B_{1},B_{2},\ldots,B_{r}$
form a copy of $K_{r}$. Since by mapping the colour class $V_{i}$
of $H$ into $B_{i}$ for each $i\in[r]$ we obtain a graph homomorphism,~(\ref{eq:Piaf})
implies that 
\begin{equation}
\sum_{i=1}^{r}\ell_{i}\mathfrak{c}(B_{i})\ge1\;.\label{eq:Polichinella}
\end{equation}

It can then be calculated that $\int_{\Omega}\mathfrak{c}\ge\frac{x}{v(H)}$,
as was needed.

\medskip{}

In the actual proof, the counterparts to common neighborhoods are
denoted $A_{i}$ and the counterparts to the smallest values of $\mathfrak{c}$
are denoted by $\alpha_{i}$. The extra difficulty coming from the
infinitesimal setting is that
\begin{enumerate}[label=(\emph{\alph*})]
\item the infimum $\alpha_{i}$ of $\mathfrak{c}$ on $A_{i}$ need not
be attained, and
\item \label{enu:noneighborhood}there is no notion of a ``cluster'',
neighborhood of which could be taken.
\end{enumerate}
A lower bound that implies that the actual sets $A_{i}$ are nonempty
is given in Claim~\ref{claim:nuAi}. In Claim~\ref{claim:IntegralFFFFFpositive}
we then show that the actual sets $B_{i}$ are indeed ``pairwise
adjacent'', thus providing a counterpart to~\ref{enu:noneighborhood}.
In Claim~\ref{claim:cover} we prove a counterpart of~(\ref{eq:Polichinella}).
These facts can be used to deduce that $\int_{\Omega}\mathfrak{c}\ge\frac{x}{v(H)}$
in a relatively straightforward way. 

\subsection{The main part of the statement\label{subsec:nonstabilityKomlos}}

We start the proof with a simple auxiliary claim.
\begin{claim}
\label{claim:normFlarge}Suppose that $t>0$, $f\in\mathcal{L}^{\infty}(\Omega)$,
$0\le f\le1$ is such that 
\[
\nu\left\{ w\in\Omega:\left\Vert W(w,\cdot)-f\right\Vert _{1}<t\right\} >0\;.
\]
Then $\|f\|_{1}\ge\delta-t$.
\end{claim}

\begin{proof}
Recall that for almost every $w\in\Omega$, we have $\|W(w,\cdot)\|_{1}\ge\delta$.
Let us fix one such $w$ which additionally satisfies $\left\Vert W(w,\cdot)-f\right\Vert _{1}<t$.
By the triangle inequality,
\[
\|f\|_{1}\ge\|W(w,\cdot)\|_{1}-\left\Vert W(w,\cdot)-f\right\Vert _{1}\ge\delta-t\;.
\]
\end{proof}
Among all proper colourings of $H$ with $r=\chi(H)$ colours consider
one that minimizes the size of the smallest colour class and let $V(H)=V_{1}\sqcup V_{2}\sqcup\ldots\sqcup V_{r}$
be the partition of the vertex set into the colour classes of this
colouring such that $\ell_{1}\ge\ell_{2}\ge\ldots\ge\ell_{r}>0,$
for $\ell_{i}=|V_{i}|$. Let $h=\sum_{i}\ell_{i}$ be the order of~$H$.
Fix an arbitrarily small~$\gamma\in(0,1)$.

Let $\mathfrak{c}:\Omega\rightarrow[0,1]$ be an arbitrary fractional
$H$-cover of $W$. It is enough to show that $\int_{\Omega}\mathfrak{c}\ge\frac{x}{v(H)}-\gamma$.
Set 
\begin{equation}
\epsilon=\gamma\cdot\left(\frac{\delta-\left(1-\frac{1}{r-1}\right)}{3r^{2}}\right)^{4}\;.\label{eq:defepsilonmetro}
\end{equation}
The fact that $x>0$ together with~(\ref{eq:defdelta}) tells us
that $\delta>1-\frac{1}{r-1}$ and $\epsilon>0$.

Let $A_{1}=\Omega$. Sequentially, for $i=1,\ldots,r$, given sets
\[
A_{1},\ldots,A_{i},B_{1},\ldots,B_{i-1},F_{1},\ldots,F_{i-1}\subset\Omega
\]
of positive measure and numbers $\alpha_{1},\ldots,\alpha_{i-1}$,
define number $\alpha_{i}$ and sets $B_{i}$, $F_{i}$, $A_{i+1}$
as follows. Set $\alpha_{i}=\ESSINF\mathfrak{c}_{\restriction A_{i}}$,
$B_{i}=\left\{ w\in A_{i}:\mathfrak{c}(w)\le\alpha_{i}+\frac{\gamma}{h}\right\} $.
It follows that $\nu(B_{i})>0$. By the separability of the space
$\mathcal{L^{\infty}}(\Omega)$ there exists a function $f_{i}\in\mathcal{L^{\infty}}(\Omega)$,
$0\le f_{i}\le1$ such that the set $F_{i}:=\left\{ w\in B_{i}:\left\Vert W(w,\cdot)-f_{i}(\cdot)\right\Vert _{1}<\epsilon\right\} $
has positive measure. Finally, define
\begin{equation}
A_{i+1}:=\left\{ w\in A_{i}:\nu\left\{ y\in F_{i}:W(w,y)>0\right\} \ge\left(1-\sqrt[4]{\epsilon}\right)\nu(F_{i})\right\} \;.\label{eq:defAi}
\end{equation}
In order to be able to proceed with the construction for step $i+1$,
we need to show that $A_{i+1}$ has positive measure. The following
claim gives an optimal quantitative lower-bound.
\begin{claim}
\label{claim:nuAi}We have $\nu(A_{i})\ge\delta-(1-\nu(A_{i-1}))-3\cdot\sqrt[4]{\epsilon}=\nu(A_{i-1})+\delta-1-3\cdot\sqrt[4]{\epsilon}$.
\end{claim}

Before proving Claim~\ref{claim:nuAi}, we note that as an immediate
consequence of Claim~\ref{claim:nuAi}, we have that 
\begin{equation}
\nu\left(A_{i+1}\right)\ge1-i\cdot(1-\delta)-3i\cdot\sqrt[4]{\epsilon}\label{eq:AiBetter}
\end{equation}
for each $i+1\le r$. Recall that $\delta>1-\frac{1}{r-1}$ by~(\ref{eq:defdelta}),
then together with~(\ref{eq:defepsilonmetro}) we know that for $i+1\le r$,
the set $A_{i+1}$ has positive measure.
\begin{proof}[Proof of Claim~\ref{claim:nuAi}]
 We want to prove that $A_{i+1}$ contains almost all of $A_{i}\cap\left(\SUPPORTPOSITIVE_{\sqrt[4]{\epsilon}}f_{i}\right)$.
To this end, we consider the quantity
\begin{equation}
\int_{w\in A_{i}\cap\left(\SUPPORTPOSITIVE_{\sqrt[4]{\epsilon}}f_{i}\right)\setminus A_{i+1}}\int_{y\in F_{i}}\left|W(w,y)-f_{i}(w)\right|=\int_{y\in F_{i}}\int_{w\in A_{i}\cap\left(\SUPPORTPOSITIVE_{\sqrt[4]{\epsilon}}f_{i}\right)\setminus A_{i+1}}\left|W(w,y)-f_{i}(w)\right|\;.\label{eq:welovefubini}
\end{equation}
First, we consider the left-hand side of~(\ref{eq:welovefubini}).
Fix $w\in A_{i}\cap\left(\SUPPORTPOSITIVE_{\sqrt[4]{\epsilon}}f_{i}\right)\setminus A_{i+1}$.
Since $w\in\SUPPORTPOSITIVE_{\sqrt[4]{\epsilon}}f_{i}$, we have $f_{i}(w)\ge\sqrt[4]{\epsilon}$.
Since $w\not\in A_{i+1}$, we have that the sets of $y\in F_{i}$,
for which $W(w,y)=0$ has measure at least $\sqrt[4]{\epsilon}\nu(F_{i})$.
Therefore, $\int_{y\in F_{i}}\left|W(w,y)-f_{i}(w)\right|\ge\sqrt[4]{\epsilon}\cdot\sqrt[4]{\epsilon}\nu(F_{i})$.
Integrating over $w$, we get
\begin{equation}
\int_{w\in A_{i}\cap\left(\SUPPORTPOSITIVE_{\sqrt[4]{\epsilon}}f_{i}\right)\setminus A_{i+1}}\int_{y\in F_{i}}\left|W(w,y)-f_{i}(w)\right|\ge\sqrt{\epsilon}\nu\left(A_{i}\cap\left(\SUPPORTPOSITIVE_{\sqrt[4]{\epsilon}}f_{i}\right)\setminus A_{i+1}\right)\nu(F_{i})\;.\label{eq:fubini1}
\end{equation}

Next, consider the right-hand side of~(\ref{eq:welovefubini}). Fix
$y\in F_{i}$. Then
\[
\int_{w\in A_{i}\cap\left(\SUPPORTPOSITIVE_{\sqrt[4]{\epsilon}}f_{i}\right)\setminus A_{i+1}}\left|W(w,y)-f_{i}(w)\right|\le\int_{w\in\Omega}\left|W(w,y)-f_{i}(w)\right|=\|W(y,\cdot)-f_{i}(\cdot)\|_{1}\le\epsilon\;,
\]
where the last inequality uses the definition of $F_{i}$. Integrating
over $y$, we get
\begin{equation}
\int_{y\in F_{i}}\int_{w\in A_{i}\cap\left(\SUPPORTPOSITIVE_{\sqrt[4]{\epsilon}}f_{i}\right)\setminus A_{i+1}}\left|W(w,y)-f_{i}(w)\right|\le\epsilon\nu(F_{i})\;.\label{eq:fubini2}
\end{equation}
Putting~(\ref{eq:fubini1}) and~(\ref{eq:fubini2}) together, we
get that 
\[
\nu\left(A_{i}\cap\left(\SUPPORTPOSITIVE_{\sqrt[4]{\epsilon}}f_{i}\right)\setminus A_{i+1}\right)\le\sqrt{\epsilon}\;.
\]
By Claim~\ref{claim:normFlarge} and the definition of $f_{i},$
we have $\|f_{i}\|_{1}\ge\delta-\epsilon$, therefore the set $\SUPPORTPOSITIVE_{\sqrt[4]{\epsilon}}f_{i}$
has measure at least $\delta-\epsilon-\sqrt[4]{\epsilon}\ge\delta-2\sqrt[4]{\epsilon}$.
Plugging these estimates into
\[
\nu(A_{i+1})\ge\nu(A_{i})-\left(1-\nu\left(\SUPPORTPOSITIVE_{\sqrt[4]{\epsilon}}f_{i}\right)\right)-\nu\left(A_{i}\cap\left(\SUPPORTPOSITIVE_{\sqrt[4]{\epsilon}}f_{i}\right)\setminus A_{i+1}\right)\;,
\]
we get the desired result.
\end{proof}
Having defined the sets $A_{1},\ldots,A_{r}$,$B_{1},\ldots,B_{r}$
and $F_{1},\ldots,F_{r}$, we want to proceed with getting control
on the numbers $\alpha_{1},\ldots,\alpha_{r}$. The following claim
is crucial to this end.
\begin{claim}
\label{claim:IntegralFFFFFpositive}We have that
\[
\int_{F_{1}\times\ldots\times F_{r}}W^{\otimes K_{r}}>0\;.
\]
\end{claim}

\begin{proof}
Note that
\[
\int_{x_{r}\in F_{r}}\int_{x_{r-1}\in F_{r-1}}\cdots\int_{x_{1}\in F_{1}}W^{\otimes K_{r}}(x_{1},\ldots x_{r})=\int_{x_{r}\in F_{r}}\int_{x_{r-1}\in N(x_{r})\cap F_{r-1}}\cdots\int_{x_{1}\in N(x_{r},x_{r-1},\ldots,x_{2})\cap F_{1}}W^{\otimes K_{r}}(x_{1},\ldots x_{r})\;.
\]
The advantage of rewriting the integral in this way is that the integrand
on the right-hand side is positive for every choice of $x_{r},\ldots,x_{1}$.
So, we only need to show that we are integrating over a set of positive
measure. Indeed, suppose that numbers $x_{r}\in F_{r}$, $x_{r-1}\in N\left(x_{r}\right)\cap F_{r-1}$,
$\ldots$, $x_{r-i}\in N\left(x_{r},\ldots,x_{r-i+1}\right)\cap F_{r-i}$
were given. It is our task to show that the measure of $N\left(x_{r},\ldots,x_{r-i}\right)\cap F_{r-i-1}$
is positive. To this end, we use that $x_{r},\ldots,x_{r-i}\in A_{r-i}$.
Then~(\ref{eq:defAi}) tells us that 
\[
\nu\left(N(x_{r})\cap F_{r-i-1}\right)\;,\;\nu\left(N(x_{r-1})\cap F_{r-i-1}\right)\;,\;\ldots\;,\;\nu\left(N(x_{r-i})\cap F_{r-i-1}\right)\;\ge\;(1-\sqrt[4]{\epsilon})\nu(F_{r-i-1})\;.
\]
We conclude that 
\[
\nu(N(x_{r},\ldots,x_{r-i})\cap F_{r-i-1})\ge(1-(i+1)\sqrt[4]{\epsilon})\nu(F_{r-i-1})>0\;,
\]
as was needed.
\end{proof}
The advertised gain of control on the numbers $\alpha_{1},\ldots,\alpha_{r}$
now follows easily.
\begin{claim}
\label{claim:cover}We have 
\begin{equation}
\ell_{1}\alpha_{1}+\ell_{2}\alpha_{2}+\ldots+\ell_{r}\alpha_{r}\ge1-\gamma.\label{eq:cover}
\end{equation}
\end{claim}

\begin{proof}
Claim~\ref{claim:IntegralFFFFFpositive} gives that $\int_{F_{1}\times\ldots\times F_{r}}W^{\otimes K_{r}}>0$.
Since $H$ is $r$-colorable, and since $F_{i}\subset B_{i}$, we
also have that 
\begin{equation}
\int_{\left(B_{1}\right)^{\ell_{1}}}\int_{\left(B_{2}\right)^{\ell_{2}}}\ldots\int_{\left(B_{r-1}\right)^{\ell_{r-1}}}\int_{\left(B_{r}\right)^{\ell_{r}}}W^{\otimes H}>0\;.\label{eq:intH}
\end{equation}
Recall that for each $w\in B_{i}$, $\mathfrak{c}(w)\le\alpha_{i}+\frac{\gamma}{h}$.
Thus, for each $\mathbf{w}\in\prod_{j}\left(B_{j}\right)^{\ell_{j}}$,
we have 
\[
\sum_{i=1}^{h}\mathfrak{c}(\mathbf{w}_{i})\le\sum_{j=1}^{r}\left(\alpha_{j}+\frac{\gamma}{h}\right)\ell_{j}=\gamma+\sum_{j=1}^{r}\ell_{j}\alpha_{j}\;.
\]
Combining~(\ref{eq:intH}) with the fact that $\mathfrak{c}$ is
a fractional $H$-cover, we get~(\ref{eq:cover}).
\end{proof}
Observe that 
\begin{eqnarray}
\int_{\Omega}\mathfrak{c} & \ge & \nu(A_{r})\alpha_{r}+\left(\nu(A_{r-1})-\nu(A_{r})\right)\alpha_{r-1}+\ldots+\left(\nu(A_{1})-\nu(A_{2})\right)\alpha_{1}\label{eq:cover1}\\
 & = & \sum_{i=2}^{r}\nu(A_{i})(\alpha_{i}-\alpha_{i-1})+\alpha_{1\;}.\nonumber 
\end{eqnarray}
Using~(\ref{eq:AiBetter}) and~(\ref{eq:cover1}) we obtain 
\[
\int_{\Omega}\mathfrak{c}\ge\sum_{i=2}^{r}\nu(A_{i})\left(\alpha_{i}-\alpha_{i-1}\right)+\alpha_{1}\ge\alpha_{1}+\sum_{i=2}^{r}\left(1-(i-1)(1-\delta)-3(i-1)\sqrt[4]{\epsilon}\right)\left(\alpha_{i}-\alpha_{i-1}\right)\,.
\]
Combined with the observation that $\sum_{i=2}^{r}\left(\alpha_{i}-\alpha_{i-1}\right)=\alpha_{r}-\alpha_{1}$,
we get
\begin{align}
\int_{\Omega}\mathfrak{c} & \ge\alpha_{r}+(\delta-1-3\sqrt[4]{\epsilon})\left(\sum_{i=2}^{r}(i-1)\left(\alpha_{i}-\alpha_{i-1}\right)\right).\nonumber \\
 & =\alpha_{r}+(\delta-1-3\sqrt[4]{\epsilon})\left((r-1)\alpha_{r}-\sum_{i=1}^{r-1}\alpha_{i}\right).\label{eq:cover2}
\end{align}
Recall that $\delta=1+x\left(\frac{1}{r-1}-\frac{1}{\chi_{\mathrm{cr}}(H)}\right)-\frac{1}{r-1}$.
Plugging this equality in~(\ref{eq:cover2}) we obtain

\begin{eqnarray}
\int_{\Omega}\mathfrak{c} & \ge & \alpha_{r}+\left(\frac{x}{r-1}-\frac{x}{\chi_{\mathrm{cr}}}-\frac{1}{r-1}-3\sqrt[4]{\epsilon}\right)\left((r-1)\alpha_{r}-\sum_{i=1}^{r-1}\alpha_{i}\right)\nonumber \\
 & \overset{\eqref{eq:cover}}{\ge} & \underbrace{\sum_{i=1}^{r-1}\frac{\alpha_{i}}{r-1}-3\sqrt[4]{\epsilon}(r-1)}_{\textsf{(R1)}}\nonumber \\
 &  & +\underbrace{\left(\frac{x}{r-1}-\frac{x}{\chi_{\mathrm{cr}}}\right)\left[\frac{r-1}{\ell_{r}}\left(1-\sum_{i=1}^{r-1}\ell_{i}\alpha_{i}-\gamma\right)-\sum_{i=1}^{r-1}\alpha_{i}\right]}_{\textsf{(R2)}}\,,\label{eq:intermediate1}
\end{eqnarray}
where we use the fact $\alpha_{r}\le1$ to get $\textsf{(R1)}$ and
use (\ref{eq:cover}) to get $\textsf{(R2)}$. Using Definition~\ref{def:chiCR},
we infer that 

\begin{equation}
\frac{x}{r-1}-\frac{x}{\chi_{\mathrm{cr}}}=x\left(\frac{1}{r-1}-\frac{h-\ell_{r}}{(r-1)h}\right)=\frac{x\ell_{r}}{(r-1)h}\,.\label{eq:weinferthat}
\end{equation}
This allows us to express the term $\textsf{(R2)}$ in~(\ref{eq:intermediate1})
as

\begin{equation}
\text{\textsf{(R2)}=}\frac{x}{h}(1-\gamma)-\frac{x}{(r-1)h}\sum_{i=1}^{r-1}\alpha_{i}\left((r-1)\ell_{i}+\ell_{r}\right)\cdot\label{eq:withX}
\end{equation}
The term $\textsf{(R1)}$ from~(\ref{eq:intermediate1}) can be decomposed
as follows:

\begin{equation}
\textsf{(R1)}=\frac{x}{(r-1)h}\sum_{i=1}^{r-1}\alpha_{i}h+\frac{1-x}{r-1}\sum_{i=1}^{r-1}\alpha_{i}-3\sqrt[4]{\epsilon}(r-1)\,.\label{eq:withoutX}
\end{equation}
Plugging the equalities~(\ref{eq:defepsilonmetro}),~(\ref{eq:withX})
and~(\ref{eq:withoutX}) in~(\ref{eq:intermediate1}) and using
the fact that $h=\sum_{i}\ell_{i}$ we get
\begin{eqnarray}
\int_{\Omega}\mathfrak{c} & = & \frac{x}{h}(1-\gamma)+\frac{x}{(r-1)h}\sum_{i=1}^{r-1}\alpha_{i}\left(h-\ell_{r}-(r-1)\ell_{i}\right)+\frac{1-x}{r-1}\sum_{i=1}^{r-1}\alpha_{i}-\sqrt[4]{\gamma}\nonumber \\
 & = & \frac{x}{h}(1-\gamma)+\frac{x}{(r-1)h}\underbrace{\sum_{i=1}^{r-1}\left(\alpha_{i}\sum_{j=1}^{r-1}\left(\ell_{j}-\ell_{i}\right)\right)}_{\mathsf{(T1)}}+\underbrace{\frac{1-x}{r-1}\sum_{i=1}^{r-1}\alpha_{i}}_{\mathsf{(T2)}}-\sqrt[4]{\gamma}\,.\label{eq:miss}
\end{eqnarray}
Let us expand the term \textsf{(T1)}. 
\begin{eqnarray*}
\sum_{i=1}^{r-1}\alpha_{i}\left(\sum_{j=1}^{r-1}\left(\ell_{j}-\ell_{i}\right)\right) & = & \sum_{i=1}^{r-1}\alpha_{i}\left[\sum_{1\le j<i}\left(\ell_{j}-\ell_{i}\right)+\sum_{i<j\le r-1}\left(\ell_{j}-\ell_{i}\right)\right]\\
 & = & \sum_{i=1}^{r-1}\sum_{j<i}(\ell_{j}-\ell_{i})(\alpha_{i}-\alpha_{j})\,.
\end{eqnarray*}
Recall that for ~$j<i$, we have~$\ell_{j}\ge\ell_{i}$ and~$\alpha_{j}\le\alpha_{i}$.
So, \textsf{(T1)} is non-negative. As~$x\le1$, we have that \textsf{(T2)}
is non-negative as well. As~$\gamma>0$ is arbitrarily small, we
obtain that $\int_{\Omega}\mathfrak{c}\ge\frac{x}{h}$ for any fractional
$H$-cover $\mathfrak{c}$. 

\subsection{Overview of the proof of the furthermore part of the statement}

Before describing the proof, let us make some observations about the
bottleneck graphon (structure of which we want to force). The only
fractional $H$-cover $\mathfrak{c}$ which satisfies $\int_{\Omega}\mathfrak{c}\le\frac{x}{v(H)}$
is constant~0 almost everywhere on $\Omega_{1}\cup\ldots\cup\Omega_{r-1}$
(using notation as described in Definition~\ref{def:bottleneckgraphon})
and constant~$\nicefrac{1}{\ell_{r}}$ almost everywhere on $\Omega_{r}$.
Also, in the idealized/discretized setting of Section~\ref{subsec:overviewMain},
the sets $A_{1}$, $A_{2}$, \ldots{}, $A_{r}$ would start with
$A_{1}=\Omega$ and then each $A_{i+1}$ would be obtained from $A_{i}$
by subtracting one set $\Omega_{\pi(i)}$ for one (but arbitrary)
permutation $\pi(1),\pi(2),\ldots,\pi(r-1)$ of $1,2,\ldots,r-1$.
In the infinitesimal setting of Section~\ref{subsec:nonstabilityKomlos},
we cannot make such a precise statement: Recall that Section~\ref{subsec:nonstabilityKomlos}
starts with fixing an error parameter $\gamma>0$, and then defining
objects based on this error parameter. Below, for a given choice of
$\gamma$, we shall denote these objects with superscript. 

So, the goal is clear on an intuitive level: if $\mathfrak{c}$ is
a fractional $H$-cover that satisfies $\int_{\Omega}\mathfrak{c}=\frac{x}{v(H)}$,
we want to describe properties of the ``limits sets'' $A_{i}^{(\gamma)}$
as $\gamma\rightarrow0$, and assert that they indeed have the same
structure as in the bottleneck graph.

The first step towards this is complementing Claim~\ref{claim:nuAi}.
Indeed, in Claim~\ref{claim:AisetminusAi} below we prove that $\nu(A_{j}^{(\gamma)}\setminus A_{j+1}^{(\gamma)})\ge1-\delta-\phi$,
where $\phi\rightarrow0$ as $\gamma\rightarrow0$. Then, in Claim~\ref{claim:EssentialRange}
we prove that the essential range of $\mathfrak{c}$ is indeed $\left\{ 0,\nicefrac{1}{\ell_{r}}\right\} $.
Now, we proceed to the key construction of the ``limits sets'' advertised
above. Namely, we define sets $O_{j}$ to be the supports of weak{*}
accumulation points the indicator functions of the sets $A_{j}^{(\gamma)}\setminus A_{j+1}^{(\gamma)}$
as $\gamma\rightarrow0$. By the discussion above, we are hoping that
the sets $O_{j}$ are the individual blocks of a bottleneck graphon.
In Claims~\ref{claim:hihiA}, \ref{claim:hihiB}, \ref{claim:Oicover}
we prove some basic properties of these sets: namely that $\nu(O_{j})\ge1-\delta$,
the sets $O_{j}$ are disjoint, and that $\mathfrak{c}$ is zero on
each $O_{j}$. In the remaining claim, the structure of $W$ is completely
forced.

\subsection{The furthermore part of the statement\label{subsec:stabilityKomlos}}

Suppose that $\FCOV(H,W)=\frac{x}{h}$ and let~$\mathfrak{c}$ be
a fractional $H$-cover attaining this value (see~(\ref{eq:fcovattained})).
For any given $\gamma>0$, we have numbers $\epsilon^{(\gamma)}$,
$\alpha_{1}^{(\gamma)},\ldots,\alpha_{r}^{(\gamma)}$, sets $A_{1}^{(\gamma)},\ldots,A_{r}^{(\gamma)}$,
$B_{1}^{(\gamma)},\ldots,B_{r}^{(\gamma)}$ and $F_{1}^{(\gamma)},\ldots,F_{r}^{(\gamma)}$,
and functions $f_{1}^{(\gamma)},\ldots,f_{r}^{(\gamma)}$ defined
in the previous part (the superscript denotes the dependence on $\gamma$). 

Since the term \textsf{(T1)} in~(\ref{eq:miss}) is non-negative,
we get from ~(\ref{eq:miss}) that

\[
\frac{x}{h}=\int_{\Omega}\mathfrak{c}\ge\frac{x}{h}(1-\gamma)-\sqrt[4]{\gamma}+\frac{1-x}{r-1}\sum_{i=1}^{r-1}\alpha_{i}^{(\gamma)}\,.
\]
This implies that 
\begin{equation}
\sum_{i=1}^{r-1}\alpha_{i}^{(\gamma)}\le\frac{2(r-1)\sqrt[4]{\gamma}}{(1-x)}\,,\label{eq:alphalittle}
\end{equation}
and consequently
\begin{equation}
\alpha_{r}^{(\gamma)}\overset{\eqref{eq:cover}}{\ge}\frac{1-\gamma-\frac{2h(r-1)\sqrt[4]{\gamma}}{(1-x)}}{\ell_{r}}\;.\label{eq:AlphaRBig}
\end{equation}
\begin{claim}
\label{claim:AisetminusAi}For any $\gamma>0$ and any $j\in[r-1]$,
we have $\nu(A_{j}^{(\gamma)}\setminus A_{j+1}^{(\gamma)})\ge1-\delta-\phi$,
where $\phi=\frac{16hr\sqrt[4]{\gamma}}{1-x}$.
\end{claim}

\begin{proof}
Let us first show that
\begin{equation}
\nu\left(A_{j+1}^{(\gamma)}\right)\le1-j(1-\delta)+\frac{\phi}{2}\;.\label{eq:Lcontr}
\end{equation}
Indeed, suppose not. Then applying Claim~\ref{claim:nuAi} repeatedly
for $i=j+2,\ldots,r-1$, we get that 
\[
\nu\left(A_{r}^{(\gamma)}\right)\ge1-(r-1)(1-\delta)+\frac{\phi}{4}\overset{\eqref{eq:weinferthat}}{\ge}\frac{x\ell_{r}}{h}+\frac{\phi}{4}\;.
\]
We then have 
\[
\int_{\Omega}\mathfrak{c}\ge\alpha_{r}^{(\gamma)}\cdot\nu\left(A_{r}^{(\gamma)}\right)\overset{\eqref{eq:AlphaRBig}}{\ge}\frac{x}{h}+\frac{\phi}{4\ell_{r}}-\frac{4r\sqrt[4]{\gamma}}{1-x}>\frac{x}{h}\;,
\]
which is a contradiction to the choice of $\mathfrak{c}$. This establishes~(\ref{eq:Lcontr}).

We have $\nu(A_{j}^{(\gamma)}\setminus A_{j+1}^{(\gamma)})=\nu(A_{j}^{(\gamma)})-\nu(A_{j+1}^{(\gamma)})$.
The measure of the former set is bounded from below by $1-(j-1)(1-\delta)-3(j-1)\cdot\sqrt[4]{\epsilon}$
by~(\ref{eq:AiBetter}), and the measure of the latter set is bounded
from above by $1-j(1-\delta)+\frac{\phi}{2}$ by (\ref{eq:Lcontr}).
The claim follows.
\end{proof}
\begin{claim}
\label{claim:EssentialRange}The essential range of $\mathfrak{c}$
is $\left\{ 0,\nicefrac{1}{\ell_{r}}\right\} $.
\end{claim}

\begin{proof}
First assume that for some $\phi>0$ there is a set~$S$ of measure
at least~$\phi$ such that $\mathfrak{c}(S)\subseteq(\phi,\frac{1}{\ell_{r}}-\phi)$.
Fix $\gamma=\left(\frac{(1-x)\phi^{2}}{2(r+1)}\right)^{4}$. Then
$\alpha_{r}^{(\gamma)}>\frac{1}{\ell_{r}}-\phi$ by~(\ref{eq:AlphaRBig}).
In particular, $S$ is disjoint from~$A_{r}^{(\gamma)}$. We get

\[
\int_{\Omega}\mathfrak{c}\ge\nu\left(S\right)\phi+\nu\left(A_{r}^{(\gamma)}\right)\alpha_{r}^{(\gamma)}\ge\phi^{2}+\left(\frac{x}{h}\cdot\ell_{r}-\sqrt[4]{\gamma}\right)\frac{1-\gamma-\frac{2h(r-1)\sqrt[4]{\gamma}}{(1-x)}}{\ell_{r}}>\frac{x}{h}\,,
\]
a contradiction. Now assume that for some $\phi>0$ there is a set~$S$
of measure at least~$\phi$ such that $\mathfrak{c}(S)\subseteq(\frac{1}{\ell_{r}}+\phi,1]$.
Fix $\gamma=\left(\frac{(1-x)\phi}{4hr}\right)^{4}$. Then

\[
\int_{\Omega}\mathfrak{c}\ge\nu\left(A_{r}^{(\gamma)}\setminus S\right)\alpha_{r}^{(\gamma)}+\nu(S)\left(\frac{1}{\ell_{r}}+\phi\right)>\frac{x}{h}\,,
\]
again a contradiction, proving the claim.
\end{proof}
Let $\left(\gamma_{n}^{(r)}\right)_{n=1}^{\infty}$ be a sequence
of numbers, with $\gamma_{n}^{(r)}\overset{n\rightarrow\infty}{\longrightarrow}0$.
Now, for a fixed $i=r-1,r-2,\ldots,1,$ we inductively derive $\left(\gamma_{n}^{(i)}\right)_{n=1}^{\infty}$
from $\left(\gamma_{n}^{(i+1)}\right)_{n=1}^{\infty}$ in the following
way. Consider the sequence of sets 
\[
\left(A_{i}^{\left(\gamma_{n}^{(i+1)}\right)}\setminus A_{i+i}^{\left(\gamma_{n}^{(i+1)}\right)}\right)_{n=1}^{\infty}
\]
viewed as indicator functions. These functions have an accumulation
point $\chi_{i}:\Omega\rightarrow[0,1]$ in the weak{*} topology by
by Theorem~\ref{thm:BanachAlauglou}. Let $O_{i}=\SUPPORT\chi_{i}$.
Let $\left(\gamma_{n}^{(i)}\right)_{n=1}^{\infty}\subset\left(\gamma_{n}^{(i+1)}\right)_{n=1}^{\infty}$
be a subsequence along which these indicator functions converge to
$\chi_{i}$. Since $O_{i}$ arises from the weak{*} limit of the sets
$A_{i}^{\left(\gamma_{n}^{(i)}\right)}\setminus A_{i+1}^{\left(\gamma_{n}^{(i)}\right)}$,
we have that 
\begin{equation}
\nu\left(O_{i}\setminus A_{i}^{\left(\gamma_{n}^{(i)}\right)}\right)=o_{n}(1)\;.\label{eq:Rise}
\end{equation}
\begin{claim}
\label{claim:hihiA}We have $\nu(O_{i})\ge1-\delta$.
\end{claim}

\begin{proof}[Proof of Claim~\ref{claim:hihiA}]
 By Claim~\ref{claim:AisetminusAi}, we have that $\nu\left(A_{i}^{\left(\gamma_{n}^{(i)}\right)}\setminus A_{i+1}^{\left(\gamma_{n}^{(i)}\right)}\right)\ge1-\delta-o_{n}(1)$.
Since $\chi_{i}$ is the weak{*} limit of the indicator functions
of the sets $A_{i}^{\left(\gamma_{n}^{(i)}\right)}\setminus A_{i+1}^{\left(\gamma_{n}^{(i)}\right)}$,
we have that 
\begin{equation}
\int_{\Omega}\chi_{i}\ge1-\delta\;.\label{eq:IntChi}
\end{equation}
Since $\ESSSUP\chi_{i}\le1$, we get that $\nu(O_{i})\ge1-\delta$.
\end{proof}
\begin{claim}
\label{claim:hihiB}The sets $O_{1}$, $O_{2}$, \ldots{}, $O_{r-1}$
are pairwise disjoint.
\end{claim}

\begin{proof}[Proof of Claim~\ref{claim:hihiB}]
 Let $i\in[r-2]$ be arbitrary. We want to show that the set $O_{i}$
is disjoint from $O_{i+1}\cup O_{i+2}\cup\ldots\cup O_{r-1}$. We
have that 
\[
\left(A_{i}^{\left(\gamma_{n}^{(i)}\right)}\setminus A_{i+1}^{\left(\gamma_{n}^{(i)}\right)}\right)\cap\left(O_{i+1}\cup O_{i+2}\cup\ldots\cup O_{r-1}\right)\subset\left(O_{i+1}\cup O_{i+2}\cup\ldots\cup O_{r-1}\right)\setminus A_{i+i}^{\left(\gamma_{n}^{(i)}\right)}\;.
\]
Recall that the support of the weak{*} limit of the indicator functions
of the sets $A_{i+1}^{\left(\gamma_{n}^{(i)}\right)}$ contains the
set $O_{i+1}\cup O_{i+2}\cup\ldots\cup O_{r-1}$. This proves the
claim.
\end{proof}
\begin{claim}
\label{claim:Oicover}The function $\mathfrak{c}_{\restriction O_{i}}$
is zero almost everywhere.
\end{claim}

\begin{proof}[Proof of Claim~\ref{claim:Oicover}]
 Suppose that this is not the case, i.e., $\mathfrak{c}$ is at least
some $\theta>0$ on a subset~$P\subset O_{i}$ of measure~$\theta$.
Recall that~$O_{i}$ arises as the weak{*} limit of the sets $A_{i}^{\left(\gamma_{n}^{(i+1)}\right)}\setminus A_{i+1}^{\left(\gamma_{n}^{(i+1)}\right)}$.
Therefore, for each $n$ sufficiently large, $\mathfrak{c}$ is at
least $\theta$ on a subset $P'\subset O_{i}\cap\left(A_{i}^{\left(\gamma_{n}^{(i+1)}\right)}\setminus A_{i+1}^{\left(\gamma_{n}^{(i+1)}\right)}\right)$
of measure~$\nicefrac{\theta}{2}$. By Claim~\ref{claim:EssentialRange},
$\mathfrak{c}_{\restriction P'}=\nicefrac{1}{\ell_{r}}$. Also, combining
Claim~\ref{claim:EssentialRange} and~(\ref{eq:AlphaRBig}) we get
that 
\[
\mathfrak{c}_{\restriction A_{r}^{\left(\gamma_{n}^{(i+1)}\right)}}=\nicefrac{1}{\ell_{r}}\;.
\]
Assume further that $n$ is such that $\gamma_{n}^{(i+1)}<\left(\frac{r^{2}\theta}{2\ell_{r}}\right)^{4}$.
Then
\begin{align*}
\int_{\Omega}\mathfrak{c} & \ge\nu\left(P'\sqcup A_{r}^{\left(\gamma_{n}^{(i+1)}\right)}\right)\cdot\frac{1}{\ell_{r}}\\
\JUSTIFY{by~\eqref{eq:defepsilonmetro}~and~\eqref{eq:AiBetter}} & \ge\left(\frac{\theta}{2}+1-(r-1)\cdot(1-\delta)-3(r-1)\cdot\sqrt[4]{\gamma_{n}^{(i+1)}}\cdot\frac{\delta-(1-\frac{1}{r-1})}{3r^{2}}\right)\cdot\frac{1}{\ell_{r}}\\
\JUSTIFY{by~\eqref{eq:defdelta}} & =\left(\frac{\theta}{2}+\frac{x\ell_{r}}{h}-\sqrt[4]{\gamma_{n}^{(i+1)}}\cdot\frac{x\ell_{r}}{r^{2}}\right)\cdot\frac{1}{\ell_{r}}>\frac{x}{h}\;,
\end{align*}
which is a contradiction to the fact that $\int_{\Omega}\mathfrak{c}=\frac{x}{h}$.
\end{proof}
We can now proceed with the inductive step for $i-1$ in the same
manner.

Having defined the functions $\chi_{i}$, the sets $O_{i}$ and the
sequences $\left(\gamma_{n}^{(i)}\right)_{n=1}^{\infty}$ for $i=r-1,\dots,1$,
we now derive some further properties of these.
\begin{claim}
\label{claim:Diana1}For $\ell=r-1,r-2,\ldots,1$ and each $j$, $\ell<j\le r-1$,
if $F_{\ell}^{\left(\gamma_{n}^{(j)}\right)}\cap O_{j}$ is not null
then $\nu\left(O_{j}\setminus F_{\ell}^{\left(\gamma_{n}^{(j)}\right)}\right)=o_{n}(1)$.
\end{claim}

\begin{claim}
\label{claim:Diana2}For $\ell=r-1,r-2,\ldots,1$ and each $j$, $\ell<j\le r-1$,
and each $n\in\mathbb{N}$ sufficiently large, we have that $F_{\ell}^{\left(\gamma_{n}^{(j)}\right)}\cap O_{j}$
is a null-set.
\end{claim}

\begin{claim}
\label{claim:hihiC}For $\ell=r-1,r-2,\ldots,1$ and for each sufficiently
large $n\in\mathbb{N}$ the set 
\[
\left(A_{\ell}^{\left(\gamma_{n}^{(\ell)}\right)}\setminus A_{\ell+1}^{\left(\gamma_{n}^{(\ell)}\right)}\right)\setminus\left(O_{\ell+1}\cup O_{\ell+2}\cup\ldots\cup O_{r-1}\cup\SUPPORT\mathfrak{c}\right)
\]
 is independent in $W.$
\end{claim}

\begin{claim}
\label{claim:hihiD}For $\ell=r-1,r-2,\ldots,1$ the set $O_{\ell}$
is independent in $W$.
\end{claim}

\begin{claim}
\label{claim:hihiCHI}For $\ell=r-1,r-2,\ldots,1$ we have that $\chi_{\ell}$
is constant $1$ almost everywhere on $O_{\ell}$ and constant $0$
almost everywhere on $\Omega\setminus O_{\ell}$.
\end{claim}

\begin{claim}
\label{claim:hihiE}For $\ell=r-1,r-2,\ldots,1$, $W$ is~1 almost
everywhere on $O_{\ell}\times\left(\Omega\setminus O_{\ell}\right)$.
\end{claim}

We shall now prove Claims~\ref{claim:Diana1}\textendash \ref{claim:hihiE}
by induction. That is, first we prove Claim~\ref{claim:Diana1},
Claim~\ref{claim:Diana2}, Claim~\ref{claim:hihiC}, Claim~\ref{claim:hihiD},
Claim~\ref{claim:hihiE} (in this order) for $\ell=r-1$, and then
continue proving the same batch of claims for $\ell=r-2,\ldots,1$.
Note that Claims~\ref{claim:Diana1} and~\ref{claim:Diana2} are
vacuous for $\ell=r-1$.
\begin{proof}[Proof of Claim~\ref{claim:Diana1}]
Suppose that $F_{\ell}^{\left(\gamma_{n}^{(j)}\right)}\cap O_{j}$
is not null. Claim~\ref{claim:hihiD} and \ref{claim:hihiE} (applied
to $\ell_{\mathrm{Cl\ref{claim:hihiD}}}=\ell_{\mathrm{Cl\ref{claim:hihiE}}}=j$)
assert that the one-variable functions $W(w,\cdot)$ are the same
for almost all $w\in O_{j}$. Consequently, 
\begin{equation}
\left\Vert W(w,\cdot)-f_{\ell}^{\left(\gamma_{n}^{(j)}\right)}(\cdot)\right\Vert _{1}<\epsilon^{\left(\gamma_{n}^{(j)}\right)}\;,\label{eq:dvere}
\end{equation}
 for almost all $w\in O_{j}$.

Combining~(\ref{eq:Rise}) with $A_{j}^{\left(\gamma_{n}^{(j)}\right)}\subset A_{\ell}^{\left(\gamma_{n}^{(j)}\right)}$,
we get 
\begin{equation}
\nu\left(O_{j}\setminus A_{\ell}^{\left(\gamma_{n}^{(j)}\right)}\right)=o_{n}(1)\;.\label{eq:OKH}
\end{equation}
By Claim~\ref{claim:Oicover}, $\mathfrak{c}$ is zero almost everywhere
on $O_{j}$. Therefore, (\ref{eq:OKH}) can be rewritten as $\nu\left(O_{j}\setminus B_{\ell}^{\left(\gamma_{n}^{(j)}\right)}\right)=o_{n}(1)$.
The claim follows by plugging~(\ref{eq:dvere}) into the definition
of $F_{\ell}^{\left(\gamma_{n}^{(j)}\right)}$.
\end{proof}

\begin{proof}[Proof of Claim~\ref{claim:Diana2}]
Suppose that the statement of the claim does not hold. Then there
exists an infinite sequence of numbers $n$ for which $F_{\ell}^{\left(\gamma_{n}^{(j)}\right)}\cap O_{j}$
is not null. Let $n$ be such that $F_{\ell}^{\left(\gamma_{n}^{(j)}\right)}\cap O_{j}$
is not null, and suppose that it is sufficiently large. We then have
that 
\[
\nu\left(O_{j}\cap F_{\ell}^{\left(\gamma_{n}^{(j)}\right)}\cap A_{j}^{\left(\gamma_{n}^{(j)}\right)}\right)\ge\nu\left(O_{j}\right)-\nu\left(O_{j}\setminus F_{\ell}^{\left(\gamma_{n}^{(j)}\right)}\right)-\nu\left(O_{j}\setminus A_{j}^{\left(\gamma_{n}^{(j)}\right)}\right)\;.
\]
The first term is at least $1-\delta$ by Claim~\ref{claim:hihiA}.
The second term is $o_{n}(1)$ by Claim~\ref{claim:Diana1}. The
third term is $o_{n}(1)$ by~(\ref{eq:Rise}). We conclude that
\begin{equation}
\nu\left(O_{j}\cap F_{\ell}^{\left(\gamma_{n}^{(j)}\right)}\cap A_{j}^{\left(\gamma_{n}^{(j)}\right)}\right)>\frac{1}{2}(1-\delta)\;.\label{eq:Topeni}
\end{equation}
Consider an arbitrary $w\in O_{j}\cap F_{\ell}^{\left(\gamma_{n}^{(j)}\right)}\cap A_{j}^{\left(\gamma_{n}^{(j)}\right)}$.
As $w\in A_{j}^{(\gamma_{n}^{(j)})}$, the definition from~(\ref{eq:defAi})
gives, 
\[
\nu\left(N(w)\cap F_{\ell}^{\left(\gamma_{n}^{(j)}\right)}\right)\ge\left(1-\sqrt[4]{\epsilon^{\left(\gamma_{n}^{(j)}\right)}}\right)\nu\left(F_{\ell}^{\left(\gamma_{n}^{(j)}\right)}\right)\;.
\]
In particular, 
\[
\nu\left(N(w)\cap O_{j}\cap F_{\ell}^{\left(\gamma_{n}^{(j)}\right)}\right)\ge\nu\left(O_{j}\cap F_{\ell}^{\left(\gamma_{n}^{(j)}\right)}\right)-\sqrt[4]{\epsilon^{\left(\gamma_{n}^{(j)}\right)}}\nu\left(F_{\ell}^{\left(\gamma_{n}^{(j)}\right)}\right)\overset{\eqref{eq:Topeni}}{\ge}\frac{1}{4}(1-\delta)\;.
\]
Integrating $w$ over the set $O_{j}\cap F_{\ell}^{\left(\gamma_{n}^{(j)}\right)}\cap A_{j}^{\left(\gamma_{n}^{(j)}\right)}$
of positive measure (by~(\ref{eq:Topeni})), and get that 
\[
\int_{w\in O_{j}\cap F_{\ell}^{\left(\gamma_{n}^{(j)}\right)}\cap A_{j}^{\left(\gamma_{n}^{(j)}\right)}}\int_{y\in O_{j}\cap F_{\ell}^{\left(\gamma_{n}^{(j)}\right)}}W(w,y)>0\;.
\]
Hence $O_{j}\cap F_{\ell}^{\left(\gamma_{n}^{(j)}\right)}$ is not
an independent set, a contradiction to Claim~\ref{claim:hihiD}.
\end{proof}

\begin{proof}[Proof of Claim~\ref{claim:hihiC}]
Suppose that the statement of the claim fails for~$\ell$. Then,
we can find two sets $P,Q\subset\left(A_{\ell}^{\left(\gamma_{n}^{(\ell)}\right)}\setminus A_{\ell+1}^{\left(\gamma_{n}^{(\ell)}\right)}\right)\setminus\left(O_{\ell+1}\cup O_{\ell+2}\cup\ldots\cup O_{r-1}\cup\SUPPORT\mathfrak{c}\right)$
such that $\int_{P\times Q}W>0$. 

Consider an $r$-tuple $\mathbf{w}\in F_{1}^{\left(\gamma_{n}^{(\ell)}\right)}\times F_{2}^{\left(\gamma_{n}^{(\ell)}\right)}\times\ldots\times F_{\ell-1}^{\left(\gamma_{n}^{(\ell)}\right)}\times P\times Q\times O_{\ell+1}\times\ldots\times O_{r-1}$.
For $j=1,2,\ldots,\ell-1$, $\mathbf{w}_{j}\in F_{j}^{\left(\gamma_{n}^{(\ell)}\right)}\subset B_{j}^{\left(\gamma_{n}^{(\ell)}\right)}\subset\mathfrak{c}^{-1}(0)$,
where the last inclusion uses (in addition to the definition of the
set $B_{j}^{\left(\gamma_{n}^{(\ell)}\right)}$) Claim~\ref{claim:EssentialRange}.
For $j=\ell,\ell+1$, we have $\mathfrak{c}(\mathbf{w}_{j})=0$ since
$P$ and $Q$ are disjoint from $\SUPPORT\mathfrak{c}$. For $j=\ell+2,\ldots,r$,
we have $\mathfrak{c}(\mathbf{w}_{j})=0$ by Claim~\ref{claim:Oicover},
except possibly a null set of exceptional values of $\mathbf{w}_{j}$.
We conclude that $\sum_{j}\mathfrak{c}\left(\mathbf{w}_{j}\right)=0$,
except possibly a null set exceptional vectors $\mathbf{w}$. In particular,
for almost every $\mathbf{w}\in F_{1}^{\left(\gamma_{n}^{(\ell)}\right)}\times F_{2}^{\left(\gamma_{n}^{(\ell)}\right)}\times\ldots\times F_{\ell-1}^{\left(\gamma_{n}^{(\ell)}\right)}\times P\times Q\times O_{\ell+1}\times\ldots\times O_{r-1}$,
\begin{equation}
v(H)\cdot\sum_{j}\mathfrak{c}\left(\mathbf{w}_{j}\right)=0\;.\label{eq:tomtp}
\end{equation}
As the chromatic number of $H$ is $r$ and each color-class of $H$
has size at most~$v(H)$, we get that the function $v(H)\cdot\mathfrak{c}$
is a fractional $K_{r}$-cover. Combined with~(\ref{eq:tomtp}),
we get that $W^{\otimes K_{r}}(\mathbf{w})=0$ (for almost every $\mathbf{w}$).
Therefore,
\begin{equation}
\int_{F_{1}^{\left(\gamma_{n}^{(\ell)}\right)}}\int_{F_{2}^{\left(\gamma_{n}^{(\ell)}\right)}}\ldots\int_{F_{\ell-1}^{\left(\gamma_{n}^{(\ell)}\right)}}\int_{P}\int_{Q}\int_{O_{\ell+1}}\ldots\int_{O_{r-1}}W^{\otimes K_{r}}=0\;.\label{eq:IntegralNula}
\end{equation}

\medskip{}

We abbreviate $\mathcal{O}=O_{\ell+1}\cup\ldots\cup O_{r-1}$. Let
us now take an arbitrary $w\in A_{\ell}^{\left(\gamma_{n}^{(\ell)}\right)}$.
Recall that $A_{\ell}^{\left(\gamma_{n}^{(\ell)}\right)}\subset A_{\ell-1}^{\left(\gamma_{n}^{(\ell)}\right)}\subset\ldots\subset A_{2}^{\left(\gamma_{n}^{(\ell)}\right)}$.
Therefore, (\ref{eq:defAi}) tells us that 
\[
\nu\left(F_{j}^{\left(\gamma_{n}^{(\ell)}\right)}\cap N(w)\right)\ge\left(1-\sqrt[4]{\epsilon^{\left(\gamma_{n}^{(\ell)}\right)}}\right)\nu\left(F_{j}^{\left(\gamma_{n}^{(\ell)}\right)}\right)
\]
for each $j=\ell-1,\ell-2,\ldots,1$. Similarly, given an arbitrary
$u_{t}\in F_{t}^{\left(\gamma_{n}^{(\ell)}\right)}$ ($t=2,\ldots,\ell-1$),
we make use of the fact that $F_{t}^{\left(\gamma_{n}^{(\ell)}\right)}\subset A_{t}^{\left(\gamma_{n}^{(\ell)}\right)}$
and deduce that 
\[
\nu\left(F_{j}^{\left(\gamma_{n}^{(\ell)}\right)}\cap N(u_{t})\right)\ge\left(1-\sqrt[4]{\epsilon^{\left(\gamma_{n}^{(\ell)}\right)}}\right)\nu\left(F_{j}^{\left(\gamma_{n}^{(\ell)}\right)}\right)
\]
 for each $j=t-1,\ell-2,\ldots,1$. Claim~\ref{claim:Diana2} tells
us that 
\[
\nu\left(F_{j}^{\left(\gamma_{n}^{(\ell)}\right)}\cap N(u_{t})\setminus\mathcal{O}\right)\ge\left(1-\sqrt[4]{\epsilon^{\left(\gamma_{n}^{(\ell)}\right)}}\right)\nu\left(F_{j}^{\left(\gamma_{n}^{(\ell)}\right)}\right)>(1-\frac{1}{2r})\nu\left(F_{j}^{\left(\gamma_{n}^{(\ell)}\right)}\right)\;.
\]
That is, starting from any $w\in A_{\ell}^{\left(\gamma_{n}^{(\ell)}\right)}$,
we can plant a positive $\nu^{\ell}$-measure of $K_{\ell}$-cliques
$wu_{\ell-1}u_{\ell-2}\ldots u_{1}$ as above. The situation is illustrated
on Figure~\ref{fig:ForcingKr}.
\begin{figure}
\includegraphics[scale=0.9]{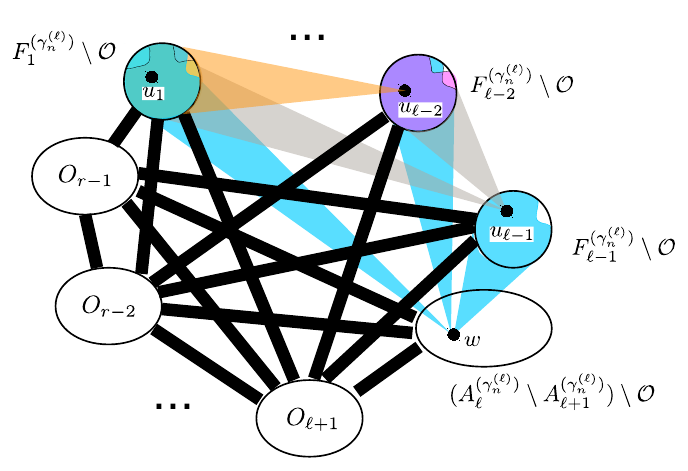}\caption{The black complete bipartite graphs are forced by Claim~\ref{claim:hihiE}.
The almost complete connections depicted with colours follow from
the fact that the respective vertices lie in the sets $A_{j}^{\left(\gamma_{n}^{(\ell)}\right)}$
($j=\ell,\ell-1,\ldots,2)$, and thus are well-connected to the sets
$F_{t}^{\left(\gamma_{n}^{(\ell)}\right)}$ (for each $t\in[j-1]$).\label{fig:ForcingKr}}
\end{figure}
 We can refine this construction to find a positive $\nu^{r}$-measure
of $K_{r}$-cliques as follows. First we take $w_{P}\in P$ and $w_{Q}\in Q$
such that $W(w_{P},w_{Q})>0$ (we have a $\nu^{2}$-positive measure
of such choices). Then we sequentially find vertices 
\[
u_{\ell-1}\in F_{\ell-1}^{\left(\gamma_{n}^{(\ell)}\right)}\setminus\mathcal{O},\ldots,u_{1}\in F_{1}^{\left(\gamma_{n}^{(\ell)}\right)}\setminus\mathcal{O}
\]
that are neighbors of $w_{P}$, $w_{Q}$ and the vertices fixed in
the previous rounds. Having chosen the $K_{\ell+1}$-clique $w_{P}w_{Q}u_{\ell-1}u_{\ell-2}\ldots u_{1}$,
Claim~\ref{claim:hihiB} tells us that $O_{\ell+1},O_{\ell+2},\ldots,O_{r-1}$
are disjoint, then together with Claim~\ref{claim:hihiE} we know
that padding arbitrary elements from $O_{\ell+1},O_{\ell+2},\ldots,O_{r-1}$
yields a copy of $K_{r}$. Since all these sets have positive measure,
we get a contradiction to~(\ref{eq:IntegralNula}).
\end{proof}

\begin{proof}[Proof of Claim~\ref{claim:hihiD}]
Recall that $O_{\ell}$ arises from the weak{*} limit of the sets
$A_{\ell}^{\left(\gamma_{n}^{(\ell)}\right)}\setminus A_{\ell+1}^{\left(\gamma_{n}^{(\ell)}\right)}$.
Claims~\ref{claim:hihiB} and~\ref{claim:Oicover} tell us that
$O_{\ell}$ can also be seen as the weak$^{*}$ limit of the sets
\[
\left(A_{\ell}^{\left(\gamma_{n}^{(\ell)}\right)}\setminus A_{\ell+1}^{\left(\gamma_{n}^{(\ell)}\right)}\right)\setminus\left(O_{\ell+1}\cup O_{\ell+2}\cup\ldots\cup O_{r-1}\cup\SUPPORT\mathfrak{c}\right)\;.
\]
Thus the claim follows by combining Claim~\ref{claim:hihiC} and
Lemma~\ref{lem:limitofindependent}.
\end{proof}

\begin{proof}[Proof of Claim~\ref{claim:hihiCHI}]
The fact that $\left(\chi_{\ell}\right)_{\Omega\setminus O_{\ell}}=0$
follows simply because $O_{\ell}$ is the indicator of $\SUPPORT\chi_{\ell}$.
Suppose now for contradiction that $\left(\chi_{\ell}\right)_{O_{\ell}}$
is less than~1 on a set of positive measure. Combining this with~(\ref{eq:IntChi})
gives that $\nu(O_{\ell})>1-\delta$.\footnote{Note this is stronger than Claim~\ref{claim:hihiA} because the inequality
is strict.} This, however cannot be the case since $\delta(W)\ge\delta$ and
$O_{\ell}$ is an independent set by Claim~\ref{claim:hihiD}.
\end{proof}

\begin{proof}[Proof of Claim~\ref{claim:hihiE}]
This follows by combining Claim~\ref{claim:hihiA}, Claim~\ref{claim:hihiD},
and the fact that the minimum degree of $W$ is at least $\delta$.
\end{proof}

\section{Comparing the proofs\label{sec:ComparingProofs}}

If not counting preparations related to the Regularity method, then
the heart of Komlós's proof of Theorem~\ref{thm:KomlosOriginal}
in~\cite{Komlos2000} is a less than three pages long calculation.
In comparison, the corresponding part of our proof in Section~\ref{subsec:nonstabilityKomlos}
has circa four pages. So, our proof is not shorter, but it is conceptually
much simpler. Indeed, Komlós's proof proceeds by an ingenious iterative
regularization of the host graph, a technique which was novel at that
time and which is rare even today (apart from proofs of variants of
Komlós's Theorem, such as~\cite{MR3471843,Grosu2012}).

Our graphon formalism, on the other hand, allows us to proceed with
the most pedestrian thinkable proof strategy. That is, to show using
relatively straightforward calculations that no small fractional $H$-covers
exist.

\medskip{}

Let us note that our proof can be de-graphonized as follows. Consider
a graph $G$ satisfying the minimum-degree condition as in~(\ref{eq:KomlosOrigMinDeg}).
Apply the min-degree form of the Regularity lemma, thus arriving to
a cluster graph~$R$. Now, the calculations from Section~\ref{subsec:nonstabilityKomlos}
can be used \emph{mutatis mutandis} to prove that~$R$ contains no
small fractional $H$-cover. Thus, by LP duality, the cluster graph~$R$
contains a large fractional $H$-tiling. This fractional $H$-tiling
in $R$ can be pulled back to a proportionally sized integral $H$-tiling
in $G$ by Blow-up lemma type techniques. The advantage of this approach
is that it allows the above mentioned argument ``take a vertex which
has the smallest value of $\mathfrak{c}$ and consider its neighborhood''
(on the level of the cluster graph), but this is compensated by the
usual technical difficulties like irregular or low density pairs.

\section{Further possible applications\label{sec:FurtherApplications}}

While Komlós's Theorem provides a complete answer (at least asymptotically)
for lower-bounding $\TIL(H,G)$ in terms of the minimum degree of
$G$, the average degree version of the problem is much less understood.
Apart from the Erd\H{o}s\textendash Gallai Theorem ($H=K_{2}$) mentioned
in Section~\ref{sec:Intro}, the only other known graphs for which
the asymptotic $F$-tiling thresholds have been determined are all
bipartite graphs,~\cite{Grosu2012} and $K_{3}$, \cite{ABHP:DensityCorHaj}.
The current graphon formalism may be of help in finding further density
thresholds.

After this paper was made public, Piguet and Saumell~\cite{EurocombMedian,PigSauMedian}
used a similar approach (with the de-graphonized formalism, as described
in Section~\ref{sec:ComparingProofs}) to obtain a strengthening
of Komlós's Theorem. In that strengthening, the lower-bound~(\ref{eq:KomlosOrigMinDeg})
is not required for all vertices but rather only for a certain (and
optimal) proportion (which depends on $x$ and the graph $H$) of
them.

Let us remark that in~\cite{DoHl:Polytons}, the authors provide
a graphon proof of the Erd\H{o}s\textendash Gallai Theorem. The key
tool to this end is to establish the half-integrality property of
the fractional vertex cover ``polyt\emph{on}''. These objects are
defined in analogy to fractional vertex cover potypes of graphs, but
for graphons (hence the ``-\emph{on}'' ending). This half-integrality
property is a direct counterpart to the well-known statement about
fractional vertex cover polytopes of finite graphs.

\section{Acknowledgments}

JH would like thank Dan Král and András Máthé for useful discussions
that preceded this project. He would also like to thank Martin Doležal
for the discussions they have had regarding functional analysis. We
thank the referees for their helpful comments.

Part of this paper was written while JH was participating in the program
\emph{Measured group theory} at The Erwin Schrödinger International
Institute for Mathematics and Physics.\medskip{}

The contents of this publication reflects only the authors' views
and not necessarily the views of the European Commission of the European
Union. This publication reflects only its authors' view; the European
Research Council Executive Agency is not responsible for any use that
may be made of the information it contains.

\bibliographystyle{plain}
\bibliography{../bibl}

\end{document}